\documentclass[a4paper,10pt]{article}
\usepackage{amsfonts,amsmath,amssymb,amstext,amsthm}
\usepackage{adjustbox}
\usepackage{array}
\usepackage{booktabs}
\usepackage{colortbl}
\usepackage[shortlabels]{enumitem}
\usepackage[left=2.5cm,top=2.4cm,right=2.5cm,bottom=2.2cm]{geometry} 
\usepackage{graphicx} \graphicspath{ {img/} }
\usepackage{hyperref} 
    \hypersetup{colorlinks, allcolors = black}
\usepackage{mathtools}
\usepackage{multirow}
\usepackage{setspace}
\usepackage{subcaption}
\usepackage{tikz}
	\usetikzlibrary{positioning}
\usepackage{url}
\usepackage{xfrac}
\usepackage{xspace}
\usepackage{xcolor}

\usepackage{caption}
\setlist{rightmargin=1em, noitemsep, topsep=2pt, partopsep=1pt}

\newtheoremstyle{result}
{6pt}{4pt}{}{}{}{}{0.5em}
{\textbf{\thmname{#1}\thmnumber{ #2}.}\thmnote{ (#3).}}

\theoremstyle{result}
\newtheorem{theorem}{Theorem}[section]
\newtheorem*{theorem*}{Theorem}
\newtheorem{proposition}[theorem]{Proposition}
\newtheorem{lemma}[theorem]{Lemma}
\newtheorem{fact}[theorem]{Fact}

\newtheorem{conjecture}[theorem]{Conjecture}

\newtheoremstyle{Definition}
{4pt}{4pt}{}{}{}{}{0.5em}
{\textbf{\thmname{#1}\thmnumber{ #2}.}\thmnote{ (#3).}}

\theoremstyle{Definition}
\newtheorem{definition}[theorem]{Definition}
\newtheorem*{definition*}{Definition}

\theoremstyle{remark}
\newtheorem{remark}[theorem]{Remark}

\newtheorem*{observation}{Observation}

\newcommand{\gluesto}{\ensuremath{\,\thicksim\,}}

\newcommand{\CF}[1]{\bigl[\,#1\,\bigr]}
\newcommand{\flatCF}[1]{[\,#1\,]}
\newcommand{\vol}{\text{vol}}

\newcommand{\myfrac}[2]{\raisebox{-0.2ex}{\scalebox{1.2}{$\sfrac{#1}{\,#2}$}}\,}

\DeclarePairedDelimiter\abs{\lvert}{\rvert}
\DeclarePairedDelimiter\norm{\lVert}{\rVert}

\newcommand{\hopf}{\ensuremath{\mathbf{H}}\xspace}

\newcommand{\cusp}{\mathfrak{c}}

\usepackage[style=numeric, hyperref=true]{biblatex}
\title{Triangulations of the `magic manifold' and families of census knots}
\author{Em K. Thompson}
\date{ }

\begin{document}
\maketitle
\onehalfspacing
\begin{abstract}
    We describe five ideal triangulations of the 3-cusped hyperbolic `magic manifold' that are each compatible with well-established techniques for triangulating Dehn fillings. Using these techniques, we construct low-complexity triangulations for all partial fillings of the magic manifold, and in particular, recover minimal triangulations for 229 of the hyperbolic census knots. Along the way, these census knots are sorted into 42 families related by twisting that can be extended indefinitely, with each member of each infinite family inheriting an upper bound on its triangulation complexity. These triangulations are conjectured to be minimal for all 42 families. 
\end{abstract}

\section{Introduction}\label{sec:intro}
The search for 3-manifold triangulations that are \textit{veritably minimal} is challenging, but alluring. Not only does a minimal triangulation measure an inherent complexity of the underlying 3-manifold, but it is typically the `best' triangulation to feed to the many 3-manifold algorithms whose running times grow exponentially with number of tetrahedra. 

In this paper we are interested in minimal triangulations of hyperbolic knot complements, and more specifically, those knots that arise from surgery on the hyperbolic 3-component chain link, 3CL. The complement of 3CL is the so-called `magic manifold' $M_3$, for which all exceptional fillings were classified by Martelli and Petronio in~\cite{Martelli-Petronio}. 

There are 1267 hyperbolic knots with minimal triangulations consisting of up to 9 ideal tetrahedra, which we herein refer to as the \emph{census knots}~\cite{BurtonCensusComplete,CallahanDeanWeeks,ckm,ckp,Dunfield}. The magic manifold has long been affiliated with hyperbolic manifolds of low complexity, and it is well-accepted that `many' of the manifolds in the SnapPea census~\cite[of hyperbolic 3-manifolds, see][]{SnapPy} can be realised by fillings of $M_3$~\cite{Dunfield,GabaiMeyerhoffMilley09,Milley09}. Since details of these realisations are not readily accessible, we provide the $M_3$ Dehn filling instructions for 229 of the census knots in Appendix~\ref{app:census-list}. In Section~\ref{sec:families}, we organise 221 of these census knots into 42 families related by twisting, highlighting many as-yet-undocumented families of census knots, which can each be extended indefinitely.

In Section~\ref{sec:triangulations}, we construct 5 different ideal triangulations of the 3-cusped manifold $M_3$ that are each compatible with the techniques to triangulate Dehn fillings using either \textit{layered solid tori} or \textit{layered chains}. Instructions for realising these triangulated Dehn fillings are described explicitly, including the necessary detail to reproduce each triangulation in Regina~\cite{Regina}. 
We confirm that the resulting triangulations achieve minimality for the identified census knots, then conjecture that the triangulations for the 42 aforementioned families continue to be minimal in general. 

\subsection{Acknowledgements}
Much of this research took place under the support of an \textit{Australian Government Research Training Program Stipend}. The author also received the \textit{Monash University Postgraduate Publications Award} to support the preparation of this manuscript, which is based on Chapter 2 of their PhD thesis.

\section{Preliminaries}\label{sec:prelims}
\subsection{Triangulation complexity}
The \textit{complexity} of a 3-manifold $M\notin\{S^3,L(3,1),\mathbb{RP}^3\}$ is typically measured by one of two dual notions: either by the minimum number of vertices in a \emph{special spine} of $M$~\cite[see ][]{MatveevTextbook}, or by the minimum number of tetrahedra in a \emph{(generalised) triangulation}. 

Upper bounds on complexity can be stated for any family of 3-manifolds that admits a well-behaved sequence of triangulations (or special spines). 
Families of knots with upper bounds implied by their well-behaved constructions include 2-bridge links~\cite{SakumaWeeks}, a family of twisted torus knots~\cite{HMPT}, and the twist knots~\cite{aribi-GPN}. The triangulations in this paper can be used to state upper bounds for any partial filling of the magic manifold. 

Jaco, Rubinstein, Spreer and Tillmann~\cite{JRT09,JRT11,JRST20,RST-cusped} have successfully verified the minimality of various infinite families of triangulations by finding lower bounds on the complexity of a manifold in terms of its homology. Interestingly, their work reveals instances of both layered solid tori and layered chains appearing in infinite families of minimal triangulations. None of these are families of knot complements, though, and when it comes to exact complexity for infinite families of hyperbolic knot complements, very little is known. 

Most well-understood are the 2-bridge links. The canonical triangulations described by Sakuma and Weeks~\cite{SakumaWeeks} are now known to \emph{not} be minimal, after Ishikawa and Nemoto~\cite{IshikawaNemoto} presented an upper bound on the complexity of all 2-bridge links. Indeed, Ben Aribi, Guéritaud and Piguet-Nakazawa~\cite{aribi-GPN} constructed triangulations of the twist knots using roughly half the number of tetrahedra suggested even by Ishikawa and Nemoto's bound.  
Notably, Ishikawa and Nemoto were the first to determine the exact complexity for an infinite family of hyperbolic knots, which they did for the 2-bridge links associated to continued fractions of the form $\CF{2,1,1,\cdots,1,1,2}$. 

An important resource in the study of triangulation complexity is the SnapPea census of hyperbolic manifolds~\cite{CalHilWeeks,HildWeeks,HodgsonWeeks1994,SnapPy,Thistle2010}, along with the database of all minimal triangulations of each, which is housed by Regina~\cite{BurtonCensusComplete,Regina}.
The sub-census of hyperbolic \textit{knots} identifies which of the manifolds in the SnapPea cusped orientable census are homeomorphic to the complement of a knot in $S^3$~\cite{BurtonCensusComplete,CallahanDeanWeeks,ckm,ckp,Dunfield}. 
This census of hyperbolic knots was built up incrementally, with knots up to complexity 6 classified first, followed by complexity 7, 8 and 9, each many years apart. As such, inter-complexity relationships between census knots are relatively under-studied.

\subsection{The magic manifold} 
From now on, we use `the magic manifold' to refer both to the minimally twisted hyperbolic 3-chain link 3CL, shown in Figure~\ref{fig:3CL-isotopy}, as well as its complement in $S^3$, which we denote by $M_3$. Note that this link and its mirror image are inequivalent as links; the orientation we use is chosen to align with Martelli and Petronio~\cite{Martelli-Petronio}. Also notice that any two components of 3CL may be interchanged, or all three components may be cyclically permuted, for example, by rotating the 3-chain diagram $D\subset\mathbb{R}^3$ (shown on the left of Figure~\ref{fig:3CL-isotopy}) around an appropriate geodesic in $\mathbb{R}^3$ (one that is either parallel or perpendicular, respectively, to the plane of projection). For concreteness, denote the three components of 3CL by $\cusp_0,\,\cusp_1$ and $\cusp_2$, with the convention that $\cusp_0$ is depicted in red, $\cusp_1$ in blue, and $\cusp_2$ in green.

\begin{figure}[htp]
    \centering
    \includegraphics[width=0.8\linewidth]{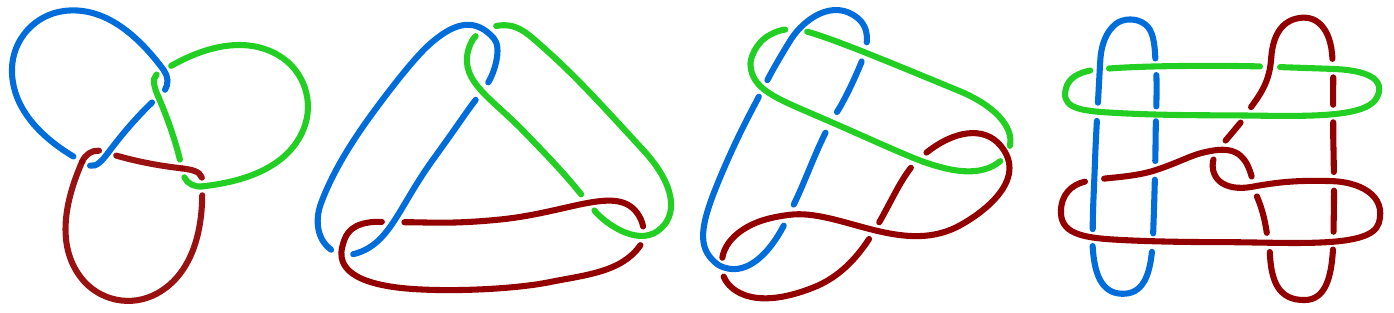}
    \caption{An isotopy of 3CL from its chain diagram to the diagram used in Figure~\ref{fig:MM-diagrams}.}
    \label{fig:3CL-isotopy}
\end{figure}

This paper relies on Martelli and Petronio's classification of all exceptional fillings of the magic manifold $M_3$ in~\cite{Martelli-Petronio}. Based on Theorem 1.1 of~\cite{Martelli-Petronio}, we distinguish the set $\mathcal{X}=\{\infty,-3,-2,-1,0\}$ of slopes that result in a non-hyperbolic filling irrespective of whether, or how, the other cusps are filled. The following theorem is a version of Theorem 1.3~\cite{Martelli-Petronio}, restated here to emphasise the fillings $M_3\big(\,\myfrac{p}{q},\,\myfrac{r}{s},\,\myfrac{t}{u}\big)=S^3$, such that $M_3\big(\,\myfrac{r}{s},\,\myfrac{t}{u}\big)$ is the complement of a hyperbolic knot in $S^3$. 

\begin{theorem}[Martelli-Petronio~\cite{Martelli-Petronio}]\label{thm:MP-classification}
    The closed manifold $M_3\big(\,\myfrac{p}{q},\,\myfrac{r}{s},\,\myfrac{t}{u}\big)$ is homeomorphic to $S^3$, and $M_3\big(\,\myfrac{r}{s},\,\myfrac{t}{u}\big)$ is a hyperbolic knot, 
    if one of the following occurs (up to permutation of $\,\myfrac{r}{s}$ and $\myfrac{t}{u}$), for $\myfrac{r}{s},\myfrac{t}{u}\notin\mathcal{X}$. 
    \begin{enumerate}[nosep, label=\Alph*., ref=\Alph*] 
        \item $\myfrac{p}{q}=\infty$, and $\myfrac{r}{s},\,\myfrac{t}{u}$ satisfy $tr-us=\pm 1$ \label{pt:firstkind}
        \item $\myfrac{p}{q}=-2$, $\,\myfrac{r}{s}=-2+\myfrac{1}{k}$ for $k\in\mathbb{Z}$, and $\myfrac{t}{u}$ satisfies $(3k-2)\,t+(6k-1)\,u=\pm 1$ \label{pt:secondkind}
        \item $\myfrac{p}{q}=-1$, $\,\myfrac{r}{s}=-3+\myfrac{1}{k}$ for $k\in\mathbb{Z}$, and $\myfrac{t}{u}$ satisfies $(2k-1)\,t+(6k-1)\,u=\pm 1$. 
    \end{enumerate}
\end{theorem}
\begin{remark}
    Herein, we refer to any knots $M_3\big(\,\myfrac{r}{s},\,\myfrac{t}{u}\big)$ as \emph{Type A, B or C}, depending on which point in Theorem~\ref{thm:MP-classification} applies.
\end{remark}
\begin{proof}
In~\cite{Martelli-Petronio}, the notation $L(p,q)$, where $\gcd(p,q)=1$, is taken to represent the lens space $L(\abs{p},q')$, with $q\equiv q'$ (mod $p$), $0<q'<\abs{p}$. In particular, if $a$ and $b$ are expressions in terms of filling slopes, an $S^3$ filling of $M_3$ can be identified in~\cite{Martelli-Petronio} by checking for solutions to $a=\pm 1$. 

Theorem~1.3 of~\cite{Martelli-Petronio} classifies all exceptional fillings of $M_3$ that result in \textit{closed} non-hyperbolic 3-manifolds. The theorem consists of four dot points. The first of these corresponds to the knots of Type A. Meanwhile, the knots of Type B and C form a subset of the $S^3$ fillings described by the second dot point, whose detailed descriptions are given in \cite[Table 2,][]{Martelli-Petronio}. 

The remaining $S^3$ fillings in \cite[Table 2,][]{Martelli-Petronio} appear as:
\begin{enumerate}[label=(\roman*)]
    \item $M_3\big(-3,-2,\,\myfrac{t}{u}\big)=L(5t+7u,2t+3u)$, \label{case:notS3a}
    \item $M_3\big(-3,-1+\myfrac{1}{n},-1+\myfrac{1}{m}\big)=L\big((2n+1)(2m+1)-4,(2n+1)m-2\big)$, and \label{case:notS3b}
    \item $M_3\big(0,n,-4-n+\myfrac{1}{m}\big)=L(6m-1,2m-1)$. \label{case:notS3c}
\end{enumerate}
Regardless of the choice for which cusp is left unfilled in~\ref{case:notS3a}, at least one of the $-2$ or $-3$ filling slopes must be used, forcing the resulting knot to be non-hyperbolic.  
Meanwhile, the knots $M_3\big(-3,-1+\sfrac{1}{n},-1+\sfrac{1}{m}\big)$ from~\ref{case:notS3b} are not hyperbolic because the possible solutions to $(2n+1)(2m+1)-4=\pm1$ (up to permutation) are $(n,m)\in\{(-1,-3),(-1,-2),(0,1),(0,2)\}$, and each of these forces one of the exceptional slopes $-2$ or $\infty$. 
For~\ref{case:notS3c}, we only see an $S^3$ filling when $m=0$, but this forces a slope of $\infty\in\mathcal{X}$, so $M_3\big(n,-4-n+\sfrac{1}{m}\big)$ is never a hyperbolic knot.

Finally, note that the third and fourth dot points in Theorem~1.3~\cite{Martelli-Petronio} refer only to exceptional fillings other than $S^3$, so their partial fillings do not describe knots. 
\end{proof}

\subsection{Organising families of knots} 
Suppose $M_3\bigl(\,\myfrac{r}{s},\myfrac{t}{u}\bigr)$ is a knot from Theorem~\ref{thm:MP-classification}, where $\myfrac{r}{s}$ is the filling slope for cusp $\cusp_1$, $\myfrac{t}{u}$ is the filling slope for cusp $\cusp_2$, and $\cusp_0$ is left unfilled. We refer to $\myfrac{r}{s}$ as the \emph{primary filling slope} and $\myfrac{t}{u}$ as the \emph{secondary filling slope} for the corresponding knot, then organise these knots into families that share a primary filling slope. 

The Type A knots are those for which the primary slope $\myfrac{r}{s}$ and secondary slope $\myfrac{t}{u}$ satisfy $tr-us=\pm 1$. 
Observe in the following standard fact that the pairs of slopes satisfying this relation can be conveniently expressed in terms of the \textit{positive continued fraction}\footnote{See~\cite{HatcherNumbers} for relevant exposition on continued fractions.} for $\myfrac{r}{s}$, where 
\vspace{-2mm}
\[ \CF{a_0,a_1, \cdots, a_{m-1}, a_m} = a_0 + \cfrac{1}{a_1+\cfrac{1}{\ \raisebox{1ex}{$\ddots$}\ a_{m-1}+\cfrac{\ 1\ }{a_m}\,.}} \]
\begin{fact}\label{fact:TypeA-sec-slopes}
    Suppose $\myfrac{r}{s}=\CF{a_0,a_1,\cdots,a_{m-1},a_m}$, with $a_1,\cdots,a_{m-1}>0$, $a_m>1$, and $a_0\in\mathbb{Z}$. Denote the penultimate convergent by $\myfrac{r'}{s'}=\CF{a_0,a_1,\cdots,a_{m-1}}$. Then, for any $\myfrac{t_n}{u_n}=\dfrac{s\cdot n+s'}{r\cdot n+r'}$, we have $t_n\cdot r-u_n\cdot s=\pm 1$.
\end{fact}

Recall that the Type B knots are those of the form $M_3\big(-2+\myfrac{1}{k},\,\myfrac{t}{u}\big)$, where $t,u$ satisfy the diophantine equation $(3k-2)\,t+(6k-1)\,u=\pm 1$. Similarly, the Type C knots are those of the form $M_3\big(-3+\myfrac{1}{k},\,\myfrac{t}{u}\big)$, with $t,u$ satisfying $(2k-1)\,t+(6k-1)\,u=\pm 1$.

\begin{lemma}\label{lem:Type-BC-slopes}
    The secondary slopes for Type B and C knots are, for any $n\in\mathbb{Z}$,
    \[ \frac{t_n}{u_n}=\frac{(2k-1)-n\,(6k-1)}{(1-k)+n\,(3k-2)} \quad \text{and} \quad \frac{t_n}{u_n}=\frac{(3k-2)-n\,(6k-1)}{(1-k)+n\,(2k-1)}, \text{ respectively.} \]
\end{lemma}
\begin{proof}
    Solutions to the diophantine equations for both Type B and Type C knots can be determined from 
    \[(2k-1)(3k-2)+(6k-1)(1-k)= 1.\]
    Note that we gain no extra information by solving the equations for $-1$, since signs cancel in the ratio of $t$ and $u$ that we are interested in. 
\end{proof}

\subsection{Knots in the Hopf complement}\label{subsec:Hopf}
The magic manifold is an example of a link complement $L=\cusp_0\cup\ldots\cup\cusp_k$ containing a \textit{Hopf sublink} $H=\cusp_i\cup\cusp_j$.  
Let $\hopf=T^2\times(-\infty,\infty)$ and let $T_h\subset\hopf$ denote the torus at height $h\in(-\infty,\infty)$. 
Note that $\hopf$ is homeomorphic to the (non-compact) complement of the Hopf link formed by $\cusp_i\cup\cusp_j$ in $S^3$.
\begin{definition}
    Define $\hopf_{L}$ to be the (non-compact) \emph{Hopf model} of a link $L=\cusp_0\cup\ldots\cup\cusp_k$ with a Hopf sublink $H=\cusp_i\cup\cusp_j$, where $\cusp_i$ is sent to $T_{+\infty}$, $\cusp_j$ is sent to $T_{-\infty}$, and the remaining components of $L-H$ are embedded in a neighbourhood of $T_{0}$. 
\end{definition}

Figure~\ref{fig:MM-diagrams} shows the relationship between the diagram of 3CL and the Hopf model $\hopf_M$ of $M_3$ where cusp $\cusp_1$ (blue) is sent to $T_{+\infty}$, cusp $\cusp_2$ (green) is sent to $T_{-\infty}$, and $\cusp_0$ is embedded in a neighbourhood of $T_{0}$. Note that in Figure~\ref{fig:MM-diagrams}, left, away from a central tangle, $\cusp_0$ runs parallel to either $\cusp_1$ or $\cusp_2$. This allows us to convert to a diagram whose projection surface is a torus, as shown in Figure~\ref{fig:MM-diagrams}, centre. On the other hand, this same figure can be interpreted as a top-down view of the Hopf model for $M_3$ shown in Figure~\ref{fig:MM-diagrams}, right.
\begin{figure}[htp]
    \centering
    \includegraphics[width=0.7\linewidth]{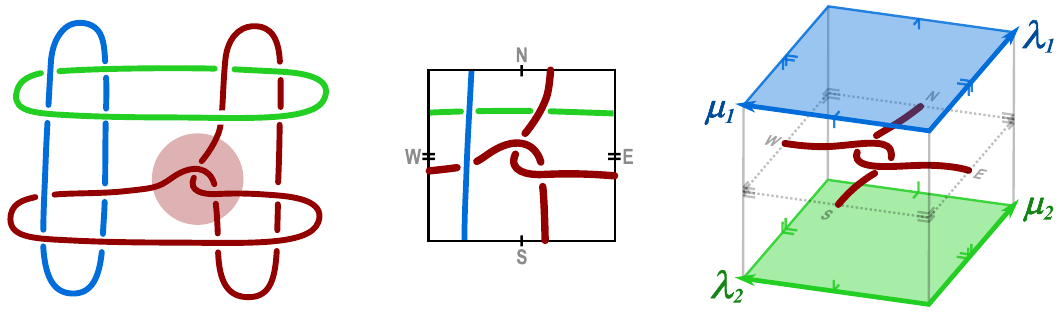}
    \caption{Left: The alternative diagram of 3CL from Figure~\ref{fig:3CL-isotopy}. Centre: A projection of the left diagram onto $T_0$, which can also be seen as a top-down view of the Hopf model on the right. Right: The (truncated) Hopf model $\hopf_M$ for $M_3$.}
    \label{fig:MM-diagrams}
\end{figure}

Note that the $\mu_i,\,\lambda_i$ framing shown for each of $\cusp_1$ and $\cusp_2$ in Figure~\ref{fig:MM-diagrams} (right) is used throughout the paper. In particular, we think of an $\myfrac{r}{s}$ Dehn filling in $\cusp_1$ as the operation that causes a geodesic $\gamma_1=r\cdot\mu_1 +s\cdot\lambda_1\subset T_{+1}$ to become homotopically trivial in the filled manifold. Similarly, a $\myfrac{t}{u}$ Dehn filling in $\cusp_2$ makes the geodesic $\gamma_2=t\cdot\mu_2 +u\cdot\lambda_2\subset T_{-1}$ homotopically trivial. 

\section{Families of census knots}\label{sec:families}
The census knots are labelled $KC_v$, with $C$ indicating the knot's triangulation complexity and $v$ ascribing an order based on the hyperbolic volumes of all equal-complexity knots. 
SnapPy~\cite{SnapPy} can be used to recognise the fillings of $M_3$ that result in census knots. By entering the diagram of $M_3$ via \emph{PLink}, we ensure the framing is as intended. It turns out that the manifold named \texttt{L6a5} has the same orientation and framing as the complement of 3CL, as shown in Figure~\ref{fig:3CL-isotopy} (left). 

Since the volume of a cusped hyperbolic manifold can only decrease after Dehn filling, no census knots with volume greater than $\mathbf{V}=\vol(M_3)=5.3335$ will appear in our list of $M_3$ fillings. Out of the 1267 census knots of complexity up to 9, there are 278 knots with volume less than $\mathbf{V}$, and we do find 229 of these by Dehn filling $M_3$. The full list is included in Appendix~\ref{app:census-list}. Any knot in this list can be entered into SnapPy as \texttt{F = Manifold(`L6a5(r,s)(t,u)')}, and its identity can be confirmed using the command \texttt{F.identify()}. This command uses the `\texttt{is\textunderscore isometric\textunderscore to}' function, so if the output is the name of a manifold \texttt{K}, then \texttt{F} and \texttt{K} are guaranteed to be isometric~\cite{SnapPy}.

\begin{remark}
There are 49 census knots with volume less than $\mathbf{V}$ that do not appear in Appendix~\ref{app:census-list}. After using SnapPy to drill out short geodesics in these knot complements without ever recovering $M_3$, we conclude that they most likely \emph{cannot} be obtained by Dehn filling the magic manifold.
\end{remark}

\subsection{Fixing a primary filling slope}
In all but one of the families that follow, we find one 9-tetrahedron knot for a secondary slope $\myfrac{t_n}{u_n}$ with index $n=-x$, and another for the secondary slope with index $n=x-1$. We define this value $x$ to be the \emph{breadth} of the family of census knots. 

We determine families of census knots that share a primary filling slope by running the following code for different values of $r_0,s_0,t_n,u_n$, according to Fact~\ref{fact:TypeA-sec-slopes} and Lemma~\ref{lem:Type-BC-slopes}. 
\begin{verbatim}
    M = Manifold(`L6a5')
    for n in range(-x,x):
        [r,s] = [r0,s0]
        [t,u] = [tn,un]
        M.dehn_fill([(0,0),(r,s),(t,u)])
        F = M.filled_triangulation()
        print(F.identify())
\end{verbatim}

Generically speaking, a knot will belong to two different families based on which slope is considered the primary slope. In practice, there are many census knots that we only see once because their second family members are too large to appear in the census. On the other hand, some knots appear more than twice, with a number of individual knots and families that can be realised by genuinely different pairs of filling slopes. These instances can be explained by the various symmetries observed by Martelli and Petronio~\cite[p.984]{Martelli-Petronio}.

We focus our attention on families of census knots that have breadth at least 2. Of the 229 knots in Appendix~\ref{app:census-list}, only 8 are \textit{not} members of such a family. In the remainder of this section we organise 221 census knots into 42 families of knots that share a primary filling slope. 

\subsubsection{Exceptional secondary slopes}
Recall that we expect to encounter some non-hyperbolic knots when a secondary slope falls in $\mathcal{X}=\{\infty,-3,-2,-1,0\}$. A secondary slope of $\infty$ always produces the unknot and a secondary slope of $0$ gives the trefoil. For secondary slopes of $-1$ and $-2$ we see various torus knots, which we label according to the fundamental group determined by SnapPy. That is, when SnapPy identifies that a non-hyperbolic filling has fundamental group of the form $\langle\,a,b\,\vert\, a^{\,p}\, b^{\,q}\,\rangle$, we assume that the knot is the $(p,q)$ torus knot and write $T_{(p,q)}$. Note that we also use $T_{(0,1)}$ to denote the unknot. 

When the secondary slope is $-3$ we obtain satellite knots. If we fix the slope $-3$ (as if it is the primary slope), then the other slope is given by the continued fraction $\CF{0,-3,m}$ for some $m\in\mathbb{Z}$. With the assistance of the software KLO~\cite{KirbyCalculator}, we find that this family is generated by $\myfrac{1}{1-m}$ fillings on the unknotted component in Figure~\ref{fig:satellite-parent}. Hence, in the tables that follow, we denote these satellite knots by $S_{-3,m}$. 
\begin{figure}
    \centering
    \includegraphics[width=0.25\linewidth]{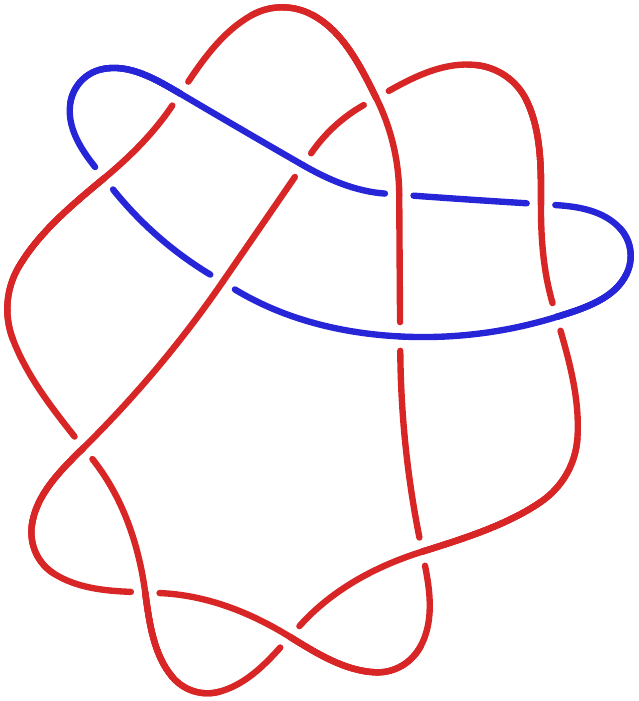}
    \caption{The 2-component parent link $M_3(-3)$ responsible for the satellite knots we encounter.}
    \label{fig:satellite-parent}
\end{figure}

\subsection{Families of note}
The families with primary slopes $\myfrac{-1}{2},\,\myfrac{1}{1},\myfrac{-4}{1}$, and $\myfrac{-5}{2}$ are well-known as fillings of the $(3/10)$ 2-bridge link, the Whitehead link, the Whitehead sister link, and the Berge manifold, respectively~\cite[see discussion on pp.1009-1011]{Martelli-Petronio}. It is noteworthy that these are precisely the 2-component links that can be built from 4 regular ideal tetrahedra~\cite{RegularTetCensus}, or from a single regular ideal octahedron~\cite{goernerPlatonicCensus}. 

These 4 families are shown in the following table. Note that the first family is the only family in this paper where we see the same census knots in the positive and negative directions, and it is the only one for which the breadth is not defined, since we find $K9_5$ at the values $n=7$ and $n=-9$ (not shown in the table).

\vspace{2mm}
\begin{adjustbox}{scale=0.9,center}
    $\begin{array}{c|c||cccccccc||c|c}
& &         0 & 1 & 2 & 3 & 4 & 5 & 6 & 7 & ~ & ~ \\ \rowcolor{black!8}
        \cellcolor{white}\multirow{-2}{*}{$(r,s)$} & \cellcolor{white}\multirow{-2}{*}{$(t_n,u_n)$} & -1 & -2 & -3 & -4 & -5 & -6 & -7 & -8 & \cellcolor{white}\multirow{-2}{*}{$M(r/s)$} & \cellcolor{white}\multirow{-2}{*}{a.k.a} \\ \hline
        ~ & ~ & T_{(2,5)} & K3_{1} & K4_{3} & K5_{5} & K6_{5} & K7_{5} & K8_{5} & K9_{5} & ~ & \cellcolor{white}\multirow{1}{*}{otet$04_{00001}$} \\ \rowcolor{black!8}
        \cellcolor{white}\multirow{-2}{*}{$(-1,2)$} & \cellcolor{white}\multirow{-2}{*}{$(-1-2 n,1+n)$} & T_{(0,1)} & T_{(2,5)} & K3_{1} & K4_{3} & K5_{5} & K6_{5} & K7_{5} & K8_{5} & \cellcolor{white}\multirow{-2}{*}{m203} & \cellcolor{white}\multirow{1}{*}{L6a2} \\ \hline
        ~ & ~ & T_{(2,3)} & K3_{2} & K4_{2} & K5_{3} & K6_{2} & K7_{2} & K8_{2} & K9_{2} & ~ & \cellcolor{white}\multirow{1}{*}{ooct$01_{00001}$} \\ \rowcolor{black!8}
        \cellcolor{white}\multirow{-2}{*}{$(1,1)$} & \cellcolor{white}\multirow{-2}{*}{$(n,1+n)$} & T_{(0,1)} & K2_{1} & K4_{1} & K5_{2} & K6_{1} & K7_{1} & K8_{1} & K9_{1} & \cellcolor{white}\multirow{-2}{*}{m129} & \cellcolor{white}\multirow{1}{*}{L5a1} \\ \hline
        ~ & ~ & K3_{1} & K5_{4} & K6_{4} & K7_{4} & K8_{4} & K9_{4} & ~ & ~ & ~ & \cellcolor{white}\multirow{1}{*}{ooct$01_{00000}$} \\ \rowcolor{black!8}
        \cellcolor{white}\multirow{-2}{*}{$(-4,1)$} & \cellcolor{white}\multirow{-2}{*}{$(-1-n,3+4 n)$} & T_{(2,3)} & K5_{1} & K6_{3} & K7_{3} & K8_{3} & K9_{3} & ~ & ~ & \cellcolor{white}\multirow{-2}{*}{m125} & \cellcolor{white}\multirow{1}{*}{L13n5885} \\ \hline      
        ~ & ~ & K4_{4} & K6_{7} & K7_{7} & K8_{7} & K9_{7} & ~ & ~ & ~ & ~ & ~ \\ \rowcolor{black!8}
        \cellcolor{white}\multirow{-2}{*}{$(-5,2)$} & \cellcolor{white}\multirow{-2}{*}{$(-1-2 n,3+5 n)$} & K3_{1} & K6_{6} & K7_{6} & K8_{6} & K9_{6} & ~ & ~ & ~ & \cellcolor{white}\multirow{-2}{*}{m202} & \cellcolor{white}\multirow{-2}{*}{otet$04_{00000}$} \\ 
    \end{array}$
\end{adjustbox}
\vspace{2mm}

The next largest families (as measured by breadth) are shown in the following table. Where the partial filling $M(\,\myfrac{r}{s})$ is a known link, we provide its name from the Hoste-Thistlethwaite census.

\vspace{2mm}
\begin{adjustbox}{scale=0.9,center}
    $\begin{array}{c|c||ccccccc||c|c}
& &         0 & 1 & 2 & 3 & 4 & 5 & 6 & ~ & ~ \\ \rowcolor{black!8}
        \cellcolor{white}\multirow{-2}{*}{$(r,s)$} & \cellcolor{white}\multirow{-2}{*}{$(t_n,u_n)$} & -1 & -2 & -3 & -4 & -5 & -6 & -7 & \cellcolor{white}\multirow{-2}{*}{$M(r/s)$} & \cellcolor{white}\multirow{-2}{*}{a.k.a} \\ \hline
        ~ & ~ & T_{(3,5)} & K4_{4} & K5_{18} & K6_{17} & K7_{30} & K8_{33} & K9_{36} & ~ & ~ \\ \rowcolor{black!8}
        \cellcolor{white}\multirow{-2}{*}{$(-1,3)$} & \cellcolor{white}\multirow{-2}{*}{$(-2-3 n,1+n)$} & T_{(0,1)} & K3_{1} & K5_{14} & K6_{14} & K7_{27} & K8_{31} & K9_{35} & \cellcolor{white}\multirow{-2}{*}{m391} & \cellcolor{white}\multirow{-2}{*}{L8a12} \\ \hline
        ~ & ~ & K2_{1} & K5_{13} & K6_{9} & K7_{11} & K8_{10} & K9_{9} & ~ & ~ & ~ \\ \rowcolor{black!8}
        \cellcolor{white}\multirow{-2}{*}{$(2,1)$} & \cellcolor{white}\multirow{-2}{*}{$(1+n,1+2 n)$} & T_{(2,3)} & K5_{9} & K6_{8} & K7_{10} & K8_{9} & K9_{8} & ~ & \cellcolor{white}\multirow{-2}{*}{m295} & \cellcolor{white}\multirow{-2}{*}{L9n14} \\ \hline
        ~ & ~ & T_{(2,7)} & K5_{7} & K6_{13} & K7_{23} & K8_{19} & K9_{20} & ~ & ~ & ~ \\ \rowcolor{black!8}
        \cellcolor{white}\multirow{-2}{*}{$(-2,3)$} & \cellcolor{white}\multirow{-2}{*}{$(-1-3 n,1+2 n)$} & T_{(3,4)} & K5_{6} & K6_{12} & K7_{22} & K8_{18} & K9_{19} & ~ & \cellcolor{white}\multirow{-2}{*}{m359} & \cellcolor{white}\multirow{-2}{*}{N/A} \\ \hline
        ~ & ~ & K3_{2} & K6_{10} & K7_{24} & K8_{22} & K9_{22} & ~ & ~ & ~ & ~ \\ \rowcolor{black!8}
        \cellcolor{white}\multirow{-2}{*}{$(1,2)$} & \cellcolor{white}\multirow{-2}{*}{$(1+2 n,1+n)$} & T_{(0,1)} & K5_{8} & K7_{18} & K8_{20} & K9_{21} & ~ & ~ & \cellcolor{white}\multirow{-2}{*}{m367} & \cellcolor{white}\multirow{-2}{*}{L7a6} \\ \hline
        ~ & ~ & S_{-3,-1} & K6_{11} & K7_{37} & K8_{49} & K9_{66} & ~ & ~ & ~ & ~ \\ \rowcolor{black!8}
        \cellcolor{white}\multirow{-2}{*}{$(-1,4)$} & \cellcolor{white}\multirow{-2}{*}{$(-3-4 n,1+n)$} & T_{(0,1)} & K5_{10} & K7_{34} & K8_{46} & K9_{61} & ~ & ~ & \cellcolor{white}\multirow{-2}{*}{s602} & \cellcolor{white}\multirow{-2}{*}{L10a114} \\ \hline
        ~ & ~ & S_{-3,2} & K6_{21} & K7_{43} & K8_{60} & K9_{81} & ~ & ~ & ~ & ~ \\ \rowcolor{black!8}
        \cellcolor{white}\multirow{-2}{*}{$(-2,5)$} & \cellcolor{white}\multirow{-2}{*}{$(-3-5 n,1+2 n)$} & T_{(3,8)} & K6_{16} & K7_{42} & K8_{59} & K9_{80} & ~ & ~ & \cellcolor{white}\multirow{-2}{*}{s661} & \cellcolor{white}\multirow{-2}{*}{N/A} \\ 
    \end{array}$
\end{adjustbox}
\vspace{2mm}

The remaining families of census knots with breadth 4, 3, 2 are summarised in Appendix~\ref{app:up-to-8}. Amongst these are three Type B and three Type C families, which are distinguished by a subscript on their primary slopes. Note that the indexing used for Type B and C knots in Appendix~\ref{app:up-to-8} is shifted in comparison to the indexing given in Lemma~\ref{lem:Type-BC-slopes}.

\subsection{Families of knots related by twisting}
As evident from Theorem~\ref{thm:MP-classification}, each family of knots can be extended indefinitely. We make the further claim that there exists a diagram for each family, such that every increase in the index $n$ corresponds to an extra full twist on some number of parallel strands. Since it is not the focus of this paper, the proposed diagrams for each family type are included in Appendix~\ref{app:knot-diagrams}.

\section{Triangulations}\label{sec:triangulations}
In this section, we describe ideal triangulations of the magic manifold by explicitly decomposing the Hopf model $\hopf_M$ into ideal tetrahedra. This is done carefully to ensure that cusps $\cusp_1$ and $\cusp_2$ can be replaced by 
either \textit{layered solid tori} or \textit{layered chains}, thus constructing triangulations for any partial filling of $M_3$. Ultimately, we construct five different triangulations, which are all necessary when accounting for minimal triangulations of census knots in Section~\ref{sec:counting-tets}.

\subsection{Triangulated Dehn fillings}\label{sec:triangulated-dehn-fillings}
We begin with a description of the techniques used to triangulate Dehn fillings. In particular, we establish the requirements for a parent triangulation to be compatible with each technique. 

\subsubsection{The Farey diagram} First, let us establish the framework we use to organise triangulations of Dehn fillings. 
\begin{definition}
    The \textit{Farey triangulation} is the triangulation of $\mathbb{H}^2$ formed by connecting two rational numbers $\myfrac{a}{b},\,\myfrac{c}{d}\in\partial\mathbb{H}^2=\mathbb{R}\cup\{\infty\}$ by a hyperbolic geodesic whenever $ad-bc=\pm1$. 
\end{definition}
\begin{definition}
    For us, the \textit{Farey diagram} is the Poincar\'e disc model of $\mathbb{H}^2$ endowed with both the Farey triangulation and its dual graph. Reference to \emph{points} and \emph{paths} in the Farey diagram mean vertices and edge-paths in the dual graph. The three rational numbers associated to the vertices of a Farey triangle form a triple that is assigned to, and used to label, the corresponding point in the diagram.
\end{definition}

We state the following facts without proof. Relevant references include~\cite{GueritaudSchleimer,HatcherNumbers,myThesis}.

\begin{fact}
The Farey diagram can be seen as the `flip graph' for one-vertex triangulations of the torus.
    \begin{enumerate}[label=\roman*.]
        \item Each point in the Farey diagram encodes a one-vertex triangulation of the torus.
        \item A step in the Farey diagram from one point to an adjacent point corresponds to `flipping an edge' in the torus triangulation.
    \end{enumerate}
\end{fact}

\begin{fact} The Farey diagram connects all rational numbers. 
    \begin{enumerate}[label=\roman*.]
        \item Every rational number appears in the Farey diagram.
        \item The dual graph of the Farey triangulation is a trivalent tree.
        \item There is a unique shortest path connecting any two points in the Farey diagram.
    \end{enumerate}
\end{fact}

\subsubsection{Layered solid tori} 
Layered solid tori, as introduced by Jaco and Rubinstein~\cite{JacoRubinstein:LST}, can be used to triangulate a Dehn filling of any given slope, provided the torus boundary to be filled has a one-vertex triangulation. 

The `layers' of a layered solid torus each consist of a single tetrahedron, with the unique 1-tetrahedron solid torus forming the innermost layer. All other layered solid tori can be constructed by \emph{layering} a sequence of tetrahedra onto the boundary of the 1-tetrahedron solid torus, with each layer adjusting the boundary triangulation by a flip. This process can also be performed `in reverse' by starting from the boundary and layering tetrahedra inwards, before forming the 1-tetrahedron solid torus at the centre by \emph{folding} the innermost boundary closed.

Note that this closing fold glues the two triangles in the one-vertex triangulation to each other, resulting in a one-triangle M\"obius band, which is the spine of the solid torus. This M\"obius band is also considered a \emph{degenerate} 0-tetrahedron layered solid torus. 

Because each layer in a layered solid torus corresponds to a flip in a one-vertex triangulation of the torus, the construction of these triangulations can be described by a path in the Farey diagram. 
Indeed, given \textit{any} one-vertex torus boundary, the Farey diagram encodes instructions for the triangulation of \textit{any} Dehn filling. This proceeds as follows. 

A \emph{starting point} in the Farey diagram is determined by the slopes of the three edges that appear in the boundary triangulation, with respect to the framing of the boundary torus. An \emph{end point} in the Farey diagram is the nearest point (to the starting point) that includes the desired Dehn filling slope. If the unique shortest path between the starting point and the end point has length $N$, then the first $N-1$ steps are used as instructions for layering on $N-1$ tetrahedra that adjust the one-vertex triangulation accordingly. The $N$th step is an instruction for how to perform the closing fold, thus making the filling slope homotopically trivial.

\subsubsection{Layered chains} 
A layered chain is a subcomplex of a triangulation described by Burton in his PhD thesis~\cite{BenThesis}. In the right setting, layered chains provide another way to triangulate Dehn fillings. Like layered solid tori, layered chains can be constructed from the outside inwards by layering some number of tetrahedra onto a torus boundary, then closing the boundary by gluing pairs of exposed faces together by folding. 

To use a layered chain for a Dehn filling, we require a two-vertex triangulation of the boundary torus. 
We call a boundary triangulation \emph{permissible} if it is isomorphic to one of the forms shown in Figure~\ref{fig:RST-bdry}. 
In order to describe fillings of a permissible boundary using a layered chain, we let $\alpha$ be the curve parallel to edge $a$ in Figure~\ref{fig:RST-bdry}, and let $\beta$ be the curve parallel to the concatenation of edges $c$ and $e$. 

\begin{figure}[htp]
    \centering
    \includegraphics[width=0.6\linewidth]{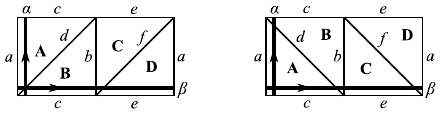}
    \caption{A \emph{permissible boundary} that a layered chain may be attached to. We say that the boundary on the left has \emph{positive diagonals}, while the boundary on the right has \emph{negative diagonals}.}
    \label{fig:RST-bdry}
\end{figure}

During the layering step of constructing a layered chain, the edges labelled $a$ and $b$ in Figure~\ref{fig:RST-bdry} must remain exposed. Indeed, we may assume that edges $c$ and $d$ also remain exposed, with all layering taking place in a `column' on top of faces \textbf{C} and \textbf{D}. 
The choice of edge to layer over is therefore restricted to $e$ or $f$, in the first instance. Layering over edge $f$ is straightforward, but to layer over edge $e$, we must first `shift' our fundamental domain so that $e$ appears as an `interior' edge (that is, cut along $f$ and reglue face \textbf{D} along $e$). Similar adjustments can be made to enable each successive step in the layering process.

At the end of the layering step we are left with a new boundary that is essentially the same, but with its two right-most faces skewed vertically. In Figure~\ref{fig:perm-fold-domain}, we use \textbf{A'} to denote the new face adjacent to \textbf{A}, and \textbf{B'} to denote the new face adjacent to \textbf{B}.

We describe the folding step only for the permissible boundary with positive diagonals, since the other case is symmetric. Take a fundamental region for the exposed torus that is isomorphic to the one shown in Figure~\ref{fig:perm-fold-domain}, left. Begin closing this boundary by gluing faces $\mathbf{B}$ and $\mathbf{B}'$ together with a fold across $b$. This identifies the edges $c\gluesto c'$ and $d\gluesto d'$ and the remaining boundary appears as in Figure~\ref{fig:perm-fold-domain}, right. Finish closing the boundary by gluing faces $\mathbf{A}$ and $\mathbf{A}'$ together with a fold across the edge $d\gluesto d'$. 
\begin{figure}[ht]
    \centering
    \includegraphics[width=0.55\linewidth]{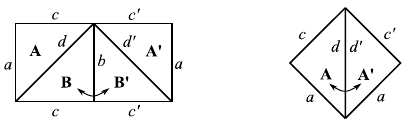}
    \caption{Instructions for closing a layered chain by folding. This shows the case where the permissible boundary has positive diagonals, and the case with negative diagonals is the same up to reflection.}
    \label{fig:perm-fold-domain}
\end{figure}

To determine the slope of a Dehn filling realised by a layered chain, we must identify the curve made homotopically trivial by the folding step, and find the slope of its isotopy class under the local cusp framing. Figures~\ref{fig:perm-slopes-a} and~\ref{fig:perm-slopes-b} demonstrate that a layered chain in a permissible boundary with slope information $\{\alpha,\beta;+\}$ uses one tetrahedron to perform a $\beta$ Dehn filling, two tetrahedra to perform a $-\alpha+\beta$ Dehn filling, no tetrahedra to perform an $\alpha+\beta$ Dehn filling, and one tetrahedron to perform a $2\cdot\alpha+\beta$ Dehn filling. 

\begin{figure}
    \centering
    \includegraphics[width=0.85\linewidth]{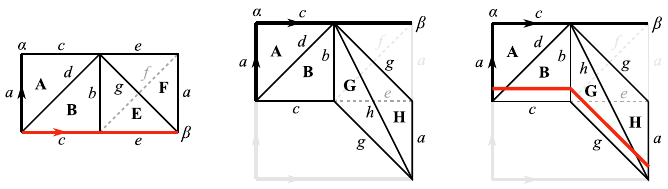}
    \caption{Left: Layer over edge $f$. Fold closed, making the red curve homotopically trivial. One tetrahedron was used and the curve made trivial was $\beta$. Centre: Instead of folding at the previous step, shift the domain and layer over edge $e$. Right: Fold closed, making the red curve homotopically trivial. Two tetrahedra were used and the curve made trivial was $-\alpha+\beta$.}
    \label{fig:perm-slopes-a}
\end{figure}

\begin{figure}
    \centering
    \includegraphics[width=0.85\linewidth]{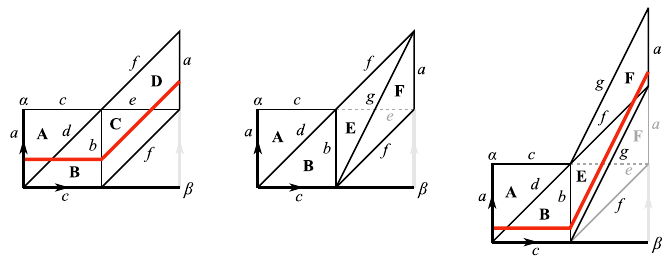}
    \caption{Left: Shift the fundamental domain, then fold closed, making the red curve homotopically trivial. No tetrahedra were used and the curve made trivial was $\alpha+\beta$. Centre: Instead of folding at the previous step, layer over edge $e$. Right: Shift the fundamental domain, then fold closed, making the red curve homotopically trivial. One tetrahedron was used and the curve made trivial was $2\cdot\alpha+\beta$.}
    \label{fig:perm-slopes-b}
\end{figure}

Note that the only curves in a permissible cusp that can be made trivial using a layered chain are those of the form $\gamma=k\cdot\alpha+\beta$, for some $k\in\mathbb{Z}$. In particular, if $\alpha$ corresponds to the meridian of the cusp, and $\beta$ corresponds to a longitude, the only Dehn fillings that can be realised by a layered chain are the integer fillings.

\subsubsection{Compatible triangulations} 
As discussed, layered solid tori and layered chains can only be used to triangulate Dehn fillings in torus boundary components equipped with one- and two-vertex triangulations, respectively. Hence, we are interested in ideal triangulations of $M_3$ that induce compatible boundary triangulations for cusps $\cusp_1$ and $\cusp_2$. 

A \textit{standard cusp} is a sub-complex homeomorphic to $T^2\times \left[0,\infty\right)$ consisting of two ideal tetrahedra, X and Y. Every standard cusp is isomorphic to the one described by the Regina gluing information below. When present in an ideal triangulation, a standard cusp can be removed to reveal a once-punctured torus boundary triangulated by two faces. Such a boundary can then be Dehn filled using a layered solid torus. Howie, Mathews and Purcell~\cite{HMP} proved that it is always possible to adjust an ideal triangulation of a $c$-cusped manifold so that $c-1$ of the cusps are standard. 

\vspace{2mm}
\begin{adjustbox}{center, scale=0.9}
\renewcommand{\arraystretch}{1.2}
\begin{tabular}{c|c|c|c|c}
    Tetrahedron & Face 012 & Face 013 & Face 023 & Face 123 \\
    \hline
    X & Y(021) & Y(031) & Y(032) & -- \\
    Y & X(021) & X(031) & X(032) & -- 
\end{tabular}
\end{adjustbox}
\vspace{2mm}

For any standard cusp $\cusp_i$, let $\mathcal{S}_i=\{s_0,s_1,s_2\}$ be the \emph{slope information} for the corresponding boundary, where each $s_i$ is the slope of one of the three edges in the one-vertex boundary, under the local framing.

We define a \emph{permissible cusp} to be the cone over a permissible boundary, as described by the Regina gluing information below. Whenever a permissible cusp is present in an ideal triangulation, we may remove the 4-tetrahedron sub-complex and perform a Dehn filling by gluing in a layered chain. 

\vspace{2mm}
\begin{adjustbox}{scale=0.9, center}
\renewcommand{\arraystretch}{1.2}
\begin{tabular}{c|c|c|c|c}
    Tetrahedron & Face 012 & Face 013 & Face 023 & Face 123 \\
    \hline 
    A & D(021) & B(031) & B(032) & -- \\
    B & C(021) & A(031) & A(032) & -- \\
    C & B(021) & D(031) & D(032) & -- \\
    D & A(021) & C(031) & C(032) & --
\end{tabular}
\end{adjustbox}
\vspace{2mm}

Using this labelling, the $\alpha$ curve is parallel to the edge A(12). In a permissible boundary with positive diagonals, the $\beta$ curve is parallel to (the concatenation of) B(31) and D(31), whereas in a permissible boundary with negative diagonals, we take $\beta$ parallel to (the concatenation of) A(23) and C(23). 
We use $\Hat{\mathcal{S}}_i=\{s_\alpha,s_\beta;\pm\}$ to encode the slope information for a permissible boundary, where $s_\alpha,s_\beta$ are the slopes of the curves $\alpha,\beta$, under the local framing, and the sign refers to the sign of the diagonals in the permissible boundary.

Suppose all tetrahedra that appear in a standard or permissible cusp in an ideal triangulation of $M_3$ are removed, leaving $t_0$ tetrahedra that we define to be its \emph{core}. 
The five triangulations of $M_3$ that we will shortly describe are determined by their respective cores, and can be distinguished by the slope information for each boundary. Hence, we shall denote them by $\mathcal{T}\big(\mathcal{S}_1,\mathcal{S}_2;\, t_0\big)$.

\subsection{Triangulating strategically}\label{sec:tri-strategy}
When decomposing a link complement into ideal tetrahedra from a diagram, we introduce ideal edges with their `endpoints' placed `on' the link. However, because the link itself is not actually present in the complement, the endpoints of an edge can be isotoped freely along the link diagram, so long as its interior never passes through a link component or another edge. Successive edges can be introduced in this way until they eventually span a collection of ideal tetrahedra that fill the entire complement. 
Note that this flexibility means copies of the same edge can appear multiple times, often in surprisingly distinct positions with respect to the diagram.

Recall the Hopf model $\hopf_M$ for $M_3$ shown in Figure~\ref{fig:MM-diagrams} (Section~\ref{subsec:Hopf}). Suppose cusp $\cusp_0$ meets the torus $T_{+1/2}$ once at the East-West identification point, and the torus $T_{-1/2}$ once at the North-South identification point, with the remainder of $\cusp_0$ contained in the interval $T^2\times (-1/2,+1/2)$. 
To describe triangulations of $\hopf_M$ that are compatible with layered solid tori and layered chains, we focus on carefully triangulating the neighbourhood of $\cusp_0$ so that any exposed faces visible from $\pm\infty$ form either a one-vertex boundary or a permissible boundary. This way we can attach either a standard/permissible cusp or a layered solid torus/layered chain, depending on context. 

We work in the universal cover of $T^2\times\mathbb{R}$, as viewed from $+\infty$. 
All triangulations we construct will have a standard cusp for $\cusp_2$, so we can always begin by introducing three ideal edges that form a one-vertex triangulation of the torus $T_{-1/2}$. 
Because we are interested in triangulations with as few tetrahedra as necessary, we attempt to `reuse' the three edges in $T_{-1/2}$ by taking isotopic copies, before introducing additional edges only as necessary.

\subsection{A first triangulation}
Introduce three edges as shown in Figure~\ref{fig:T1-cusp2}. We see two copies of the red and orange edges, which glue to form a one-vertex triangulation of $T_{-1/2}$. Coning down to $-\infty$ forms a standard cusp for $\cusp_2$, as shown in Figure~\ref{fig:T1-cusp2}.

\begin{figure}[h]
    \centering
    \begin{subfigure}[t]{0.45\textwidth}
    \centering
    \includegraphics[width=0.9\linewidth]{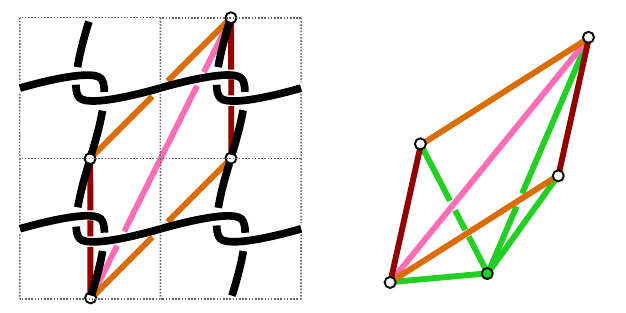}
    \caption{The three edges below the plane form a one-vertex triangulation of $T_{-1/2}$, from which we cone down to $\cusp_2$ to form a standard cusp.}
    \label{fig:T1-cusp2}
    \end{subfigure}
    \begin{subfigure}[t]{0.45\textwidth}
    \centering
    \includegraphics[width=0.9\linewidth]{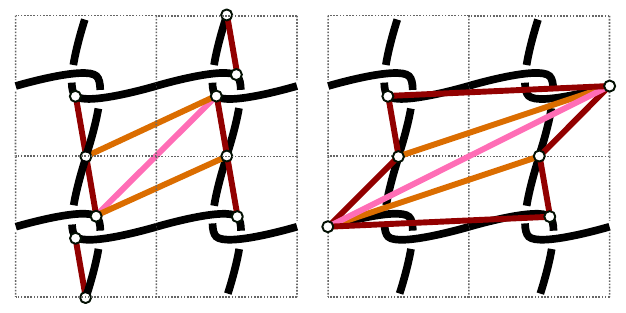}
    \caption{The three edges are isotoped along cusp $\cusp_0$ to be reused in new positions. Extra copies of the red edge are included to form two more triangular faces.}
    \label{fig:T1-isotopy}
    \end{subfigure}
    \caption{The first steps in constructing triangulation $\mathbf{T}_1$.}
    \label{fig:T1-part1}
\end{figure}

Next, we drag the endpoints of these edges along $\cusp_0$ and see that the pink edge can be pulled through so that it lies completely in $T_{+1/2}$. In Figure~\ref{fig:T1-isotopy}, we have included additional copies of the red edge, which are all either isotopic to each other or correspond to different lifts in the universal cover. By introducing one more edge (the teal edge in Figure~\ref{fig:T1-centre-tets}) we form two new tetrahedra. At this point we have a one-vertex triangulation of $T_{+1/2}$ (see Figure~\ref{fig:T1-cusp1}), so we can cone up to $\cusp_1$, forming another standard cusp. 

\begin{figure}[h]
    \centering
    \begin{subfigure}[t]{0.45\textwidth}
    \centering
    \includegraphics[width=0.9\linewidth]{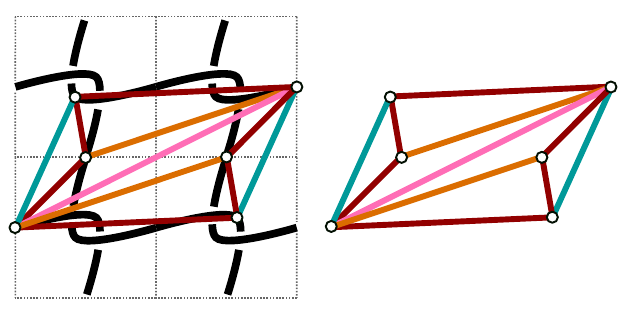}
    \caption{By adding only one extra edge, we form two more tetrahedra. Note that the two `inner' vertices in this figure are located below the others on $T_{-1/2}$.}
    \label{fig:T1-centre-tets}
    \end{subfigure}
    \begin{subfigure}[t]{0.45\textwidth}
    \centering
    \includegraphics[width=0.9\linewidth]{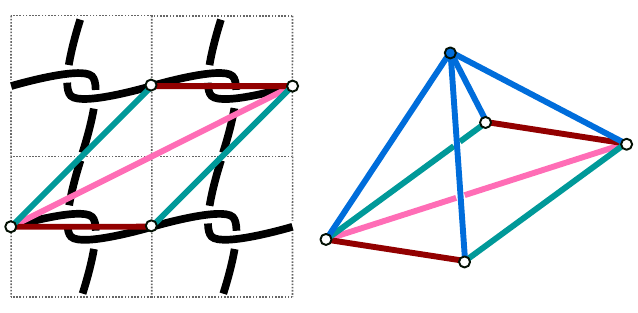}
    \centering
    \caption{The top faces of the new tetrahedra form a one-vertex triangulation of $T_{+1/2}$, and so we may cone upwards to cusp $\cusp_1$, forming our final two tetrahedra.}
    \label{fig:T1-cusp1}
    \end{subfigure}
    \caption{Forming the centre tetrahedra and cusp $\cusp_1$ for $\mathbf{T}_1$.}
    \label{fig:T1-part2}
\end{figure}

Figure~\ref{fig:T1-all-tets} shows all six tetrahedra just described, with labels added for use in Regina. The corresponding Regina gluing table, labelled $\mathbf{T}_1$, is shown below. Note that the faces 0(012), 1(210), 0(013) and 1(031) were not implicitly paired in our construction, however, by taking different lifts of the tetrahedra in Figure~\ref{fig:T1-centre-tets} we can see that 0(012) glues to 1(210) via the East-West identification, and 0(013) glues to 1(031) via the North-South identification. Regina recognises this triangulation as \texttt{s776 \#3}, which is the third minimal triangulation of the magic manifold (using its SnapPea census name). 

\begin{figure}[h]
    \centering
    \includegraphics[width=0.7\linewidth]{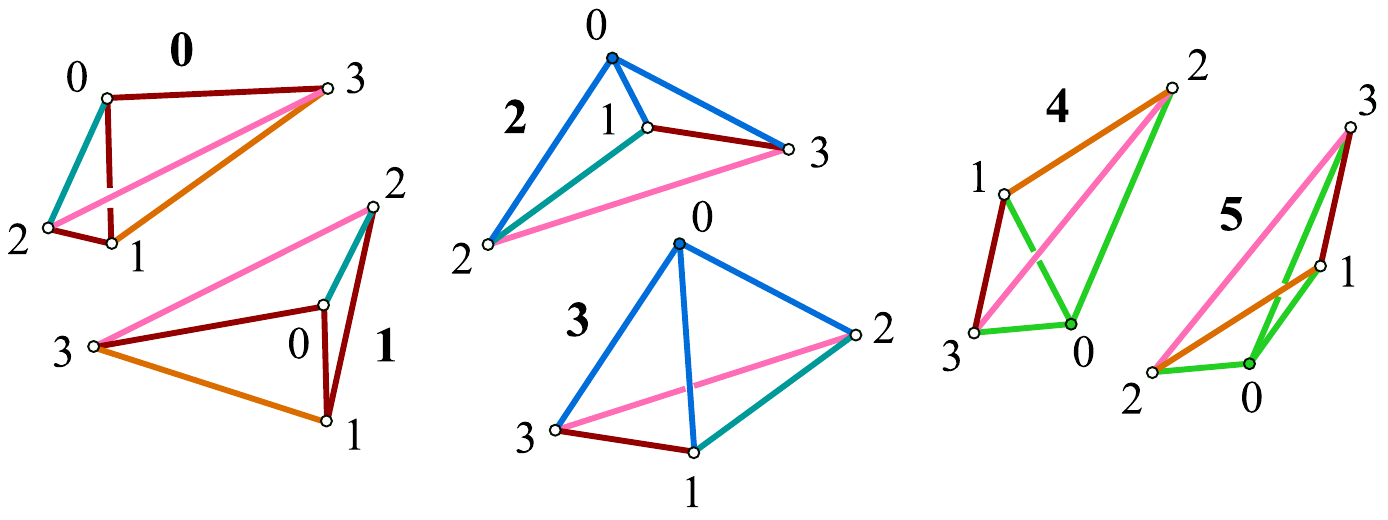}
    \caption{All tetrahedra in the first triangulation of $M_3$, with edges coloured as in Figures~\ref{fig:T1-part1} and~\ref{fig:T1-part2}, and vertices labelled as in the gluing table for $\mathbf{T}_1$.} 
    \label{fig:T1-all-tets}
\end{figure}

\vspace{2mm}
\begin{adjustbox}{scale=0.9, center}
\renewcommand{\arraystretch}{1.2}
\begin{tabular}{c|c|c|c|c}
    $\mathbf{T}_1$ & Face 012 & Face 013 & Face 023 & Face 123 \\
    \hline
    0 & 1(210) & 1(031) & 2(123) & 4(132) \\
    1 & 0(210) & 0(031) & 3(123) & 5(132) \\
    2 & 3(021) & 3(031) & 3(032) & 0(023) \\
    3 & 2(021) & 2(031) & 2(032) & 1(023) \\
    4 & 5(021) & 5(031) & 5(032) & 0(132) \\
    5 & 4(021) & 4(031) & 4(032) & 1(132)
\end{tabular}
\end{adjustbox}
\vspace{2mm}

If we kept $T_{-1/2}$ and $T_{+1/2}$ exposed, instead of coning to each of $\cusp_1$ and $\cusp_2$, we would be left with a core consisting of the two tetrahedra labelled 0 and 1. In Figure~\ref{fig:T1-slopes} we see that the boundary edges visible from cusp $\cusp_1$ have slopes $\left\{\myfrac{1}{0},\myfrac{-1}{1},\myfrac{-2}{1}\right\}$ (under the framing of $\cusp_1$), while the boundary edges visible from cusp $\cusp_2$ also have slopes $\left\{\myfrac{1}{0},\myfrac{-1}{1},\myfrac{-2}{1}\right\}$ (though now with respect to the framing of $\cusp_2$). 

We let $\mathcal{P}_i=\left\{\myfrac{1}{0},\myfrac{-1}{1},\myfrac{-2}{1}\right\}$ and write this first triangulation as $\mathbf{T}_1=\mathcal{T}(\mathcal{P}_1,\mathcal{P}_2;2)$.

\begin{figure}[!h]
    \centering
    \includegraphics[width=0.45\linewidth]{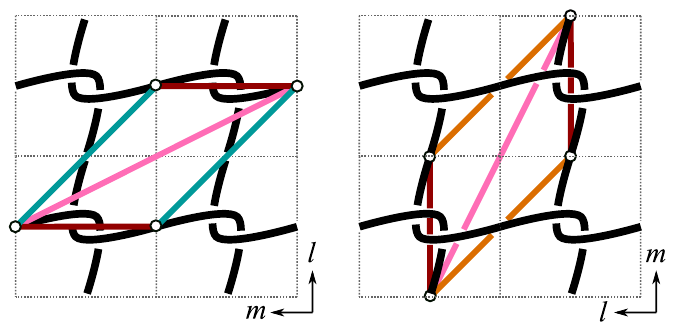}
    \caption{The boundary edges for $\mathbf{T}_1$. In cusp $\cusp_1$ (left), the red edge is 1(03) and has slope $1/0$, the teal edge is 1(02) and has slope $-1/1$, and the pink edge is 1(32) and has slope $-2/1$. In cusp $\cusp_2$ (right), the red and pink edges are 1(12) and 1(23) and also have slopes $1/0$ and $-2/1$ with respect to the $\cusp_2$ framing, and the orange edge is 1(13) and has slope $-1/1$.}
    \label{fig:T1-slopes}
\end{figure}

\subsection{Additional triangulations}
In the next triangulations, we denote the slope information for one-vertex boundaries by $\mathcal{P}_i,\,\mathcal{Q}_i$ or $\mathcal{R}_i$ with
\vspace{-3mm}
\[\mathcal{P}_i=\left\{\myfrac{1}{0},\myfrac{-1}{1},\myfrac{-2}{1}\right\}, \ \ \mathcal{Q}_i=\left\{\myfrac{1}{0},\myfrac{-2}{1},\myfrac{-3}{1}\right\}, \  \text{ and } \  \mathcal{R}_i=\left\{\myfrac{1}{0},\,\myfrac{1}{1},\,\myfrac{0}{1}\right\}.\] 

\subsubsection{Triangulation 2}
The second triangulation we construct begins with the same edges in $T_{-1/2}$, with slopes $\myfrac{1}{0},\myfrac{-1}{1}$, and $\myfrac{-2}{1}$. We proceed exactly as before, until we come to introducing a new edge in $T_{+1/2}$. This time we choose the opposite diagonal, placing the teal edge as in Figure~\ref{fig:T2-edges} where it has slope $\myfrac{-3}{1}$ (instead of $\myfrac{-1}{1}$). 
In this way we obtain $\mathbf{T}_2=\mathcal{T}\big(\mathcal{Q}_1,\mathcal{P}_2; 2\big)$, with gluing information as below. 
The labelling is such that 0(03)\gluesto 1(30) represents the teal edge (with slope $\myfrac{-3}{1}$ in $\cusp_1$) and 0(23)\gluesto 1(32) represents the pink edge (which has slope $\myfrac{-2}{1}$ in each respective cusp). We use the labels $\text{X}_1,\text{Y}_1,\text{X}_2,\text{Y}_2$ to indicate the faces where the two tetrahedra of a standard cusp $\text{X}_i\cup \text{Y}_i$ may be attached. 

\vspace{2mm}
\begin{adjustbox}{scale=0.9, center}
\renewcommand{\arraystretch}{1.2}
\begin{tabular}{c|c|c|c|c}
    $\mathbf{T}_2$ & Face 012 & Face 013 & Face 023 & Face 123 \\
    \hline
    0 & 1(102) & 1(310) & $\text{X}_1$ & $\text{X}_2$ \\ 
    1 & 0(102) & 0(310) & $\text{Y}_1$ & $\text{Y}_2$ 
\end{tabular}
\end{adjustbox}

\begin{figure}[!h]
    \centering
    \includegraphics[width=0.8\linewidth]{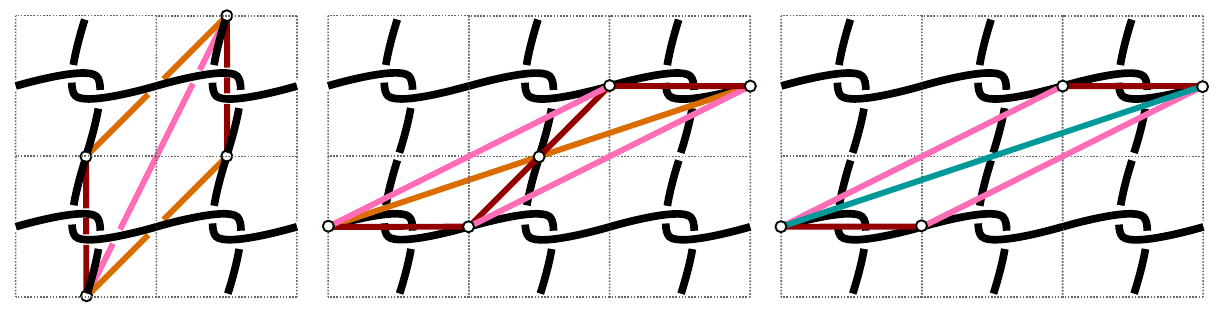}
    \caption{For $\mathbf{T}_2$, we start with the same edges on $T_{-1/2}$ but this time we consider different lifts in the universal cover, and introduce the teal edge so that it now has slope $-3/1$ with respect to the framing of $\cusp_1$.}
    \label{fig:T2-edges}
\end{figure}

\subsubsection{Triangulation 3}
The third triangulation we construct has two standard cusps with non-negative boundary slopes. 
We start by introducing three edges on $T_{-1/2}$ with slopes $\myfrac{0}{1},\,\myfrac{1}{1},\, \myfrac{1}{0}$. We form the first two tetrahedra by coning  \emph{into} the point on $\cusp_0$ that meets $T_{+1/2}$, as shown in Figure~\ref{fig:T3-tets}, left. This requires two new edges to be introduced (coloured teal and purple in the figure).

\begin{figure}[!h]
    \centering
    \includegraphics[width=0.72\linewidth]{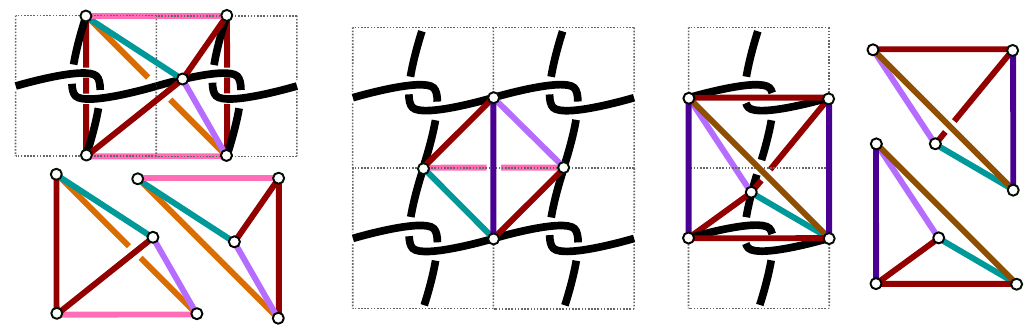}
    \caption{Edges forming the 5-tetrahedron core of $\mathbf{T}_3$.}
    \label{fig:T3-tets}
\end{figure}

In this instance, we repeat the same steps working inwards from $T_{+1/2}$. The one-vertex triangulation of $T_{+1/2}$ reuses the red edge, which has slope $\myfrac{1}{0}$ with respect to the framing of $\cusp_1$, and we introduce 2 new edges with slopes $\myfrac{0}{1}$ and $\myfrac{1}{1}$. We then cone down to the point on $\cusp_0$ that meets $T_{-1/2}$, this time reusing the teal and purple edges (see Figure~\ref{fig:T3-tets}, right). At this point, notice that the existing edges also span an additional tetrahedron shown in Figure~\ref{fig:T3-tets}, centre. 
These 5 tetrahedra form $\mathbf{T}_3=\mathcal{T}\big(\mathcal{R}_1,\mathcal{R}_2; 5\big)$, as described in Figure~\ref{fig:T3-Regina}. 
\begin{figure}[h]
\centering
\includegraphics[width=0.45\linewidth]{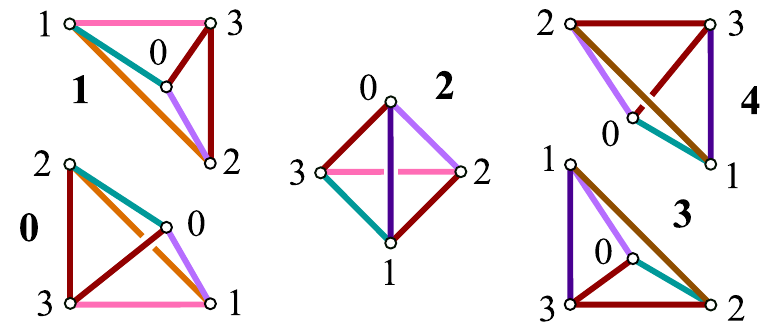} \quad 
\begin{adjustbox}{scale=0.9, valign=b}
\renewcommand{\arraystretch}{1.2}
\begin{tabular}{c|c|c|c|c}
    $\mathbf{T}_3$ & Face 012 & Face 013 & Face 023 & Face 123 \\
    \hline
    0 & 1(021) & 2(023) & 3(203) & $\text{X}_2$ \\ 
    1 & 0(021) & 2(132) & 4(203) & $\text{Y}_2$ \\ 
    2 & 3(130) & 4(310) & 0(013) & 1(031) \\
    3 & 4(021) & 2(201) & 0(023) & $\text{X}_1$ \\ 
    4 & 3(021) & 2(310) & 1(203) & $\text{Y}_1$ 
\end{tabular}
\end{adjustbox}
\vspace{2mm}
    \caption{$\mathbf{T}_3=\mathcal{T}\big(\mathcal{R}_1,\mathcal{R}_2; 5\big)$: recognised as \texttt{s776} after adding cusps and simplifying. }
    \label{fig:T3-Regina}
\end{figure}

By attaching standard cusps to both boundaries (labelled $\text{X}_i,\text{Y}_i$ in the gluing table), we form a 9-tetrahedron triangulation of $M_3$. While this triangulation is not immediately recognisable to Regina (it is clearly not minimal), it can be simplified then recognised as \texttt{s776}.

\subsubsection{Triangulation 4}
The next triangulation has one permissible cusp and one standard cusp. 
Recall that $\mathcal{R}$ is the set of slopes $\left\{\myfrac{1}{0},\,\myfrac{1}{1},\,\myfrac{0}{1}\right\}$. Let $\mathcal{R}'=\left\{\myfrac{1}{0},\myfrac{-1}{1},\,\myfrac{0}{1}\right\}$. 
In addition, let $\Hat{\mathcal{V}}=\{\myfrac{1}{0},\,\myfrac{1}{1};\,-\}$ and $\ \Hat{\mathcal{U}}=\{\myfrac{1}{0},\myfrac{-5}{1};\,+\}$ each be the slope information for a permissible cusp.

In Figure~\ref{tab:T4-tets}, we see the two tetrahedra from $\mathbf{T}_3$ that formed the one-vertex boundary for $\cusp_2$. By taking appropriate lifts of the four `top' faces of these two tetrahedra, we can identify a permissible boundary with negative diagonals, as shown to the right. From this diagram, we see that the $\alpha$ curve, which is parallel to the red edge 0(03), has slope $\myfrac{1}{0}$ with respect to the $\cusp_1$ framing. Meanwhile, the $\beta$ curve, which is parallel to the concatenation of the teal edge 0(02) and the purple edge 0(10), has slope $\myfrac{1}{1}$. The gluing table for this triangulation, denoted $\Hat{\mathbf{T}}_4=\mathcal{T}\big(\Hat{\mathcal{V}}_1,\mathcal{R}_2;\,2 \big)$, is given below, with X, Y indicating the faces of the one-vertex boundary, and A, B, C, D indicating the faces of the permissible boundary, matching the labels from Figure~\ref{fig:RST-bdry} (Section~\ref{sec:triangulated-dehn-fillings}). 

\begin{figure}[h]
\centering
 \includegraphics[width=0.45\linewidth]{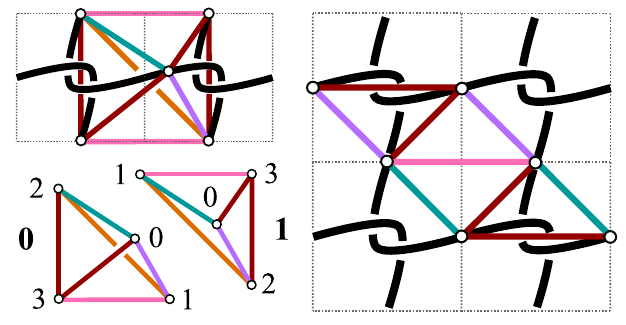}   
\caption{$\Hat{\mathbf{T}}_4=\mathcal{T}\big(\Hat{\mathcal{V}}_1,\mathcal{R}_2;\,2 \big)$: recognised as \texttt{s776} after adding cusps and simplifying.}
    \label{tab:T4-tets}
\end{figure}

\vspace{2mm}
\begin{adjustbox}{scale=0.9, center}
\renewcommand{\arraystretch}{1.2}
\begin{tabular}{c|c|c|c|c}
    $\Hat{\mathbf{T}}_4$ & Face 012 & Face 013 & Face 023 & Face 123 \\
    \hline 
    0 & 1(021) & C & A & X \\
    1 & 0(021) & B & D & Y \\
\end{tabular}
\end{adjustbox}

\subsubsection{Triangulation 5}
The final triangulation we describe is $\Hat{\mathbf{T}}_5=\mathcal{T}\big(\Hat{\mathcal{U}}_1,\mathcal{R}'_2;\,2\big)$, with tetrahedra as shown in the top of Figure~\ref{fig:T5}. 
A considerable amount of skewing is required to visualise the $\cusp_1$ and $\cusp_2$ boundaries, as we see in the bottom of Figure~\ref{fig:T5}. The red, pink and orange edges form the boundary triangulation for cusp $\cusp_2$. It is relatively easy to see (from the dashed edges in the diagram) that their slopes are indeed $\myfrac{1}{0},\,\myfrac{0}{1}$ and $\myfrac{-1}{1}$, respectively, with respect to the framing of $\cusp_2$.

\begin{figure}[h]
    \centering
    \includegraphics[width=0.6\linewidth]{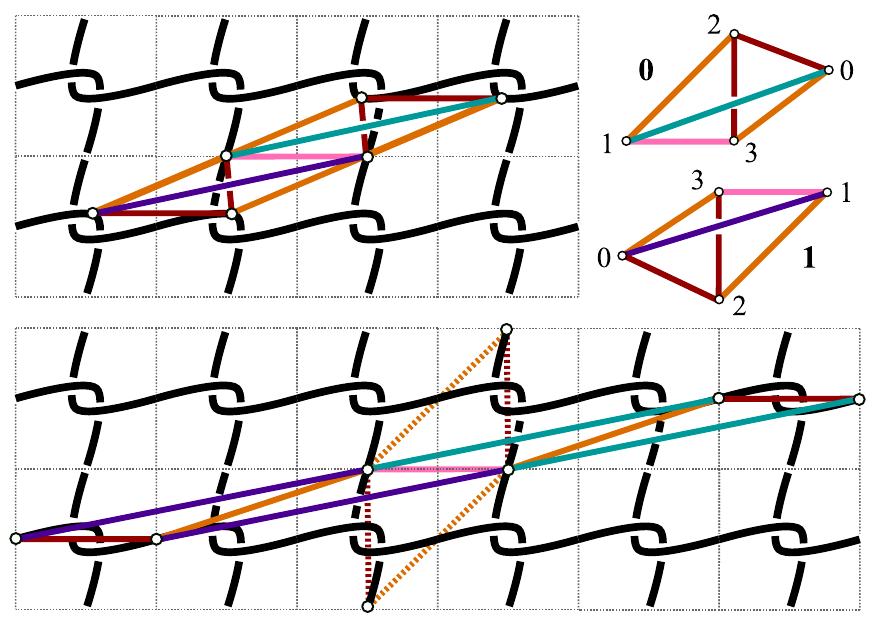}
    \caption{The two core tetrahedra in $\Hat{\mathbf{T}}_5$, with labelling used in the Regina gluing table.} 
    \label{fig:T5}
\end{figure}

For the permissible boundary in $\cusp_1$, we first skew the tetrahedra in Figure~\ref{fig:T5} (top) by sliding the top-right red edge further to the right and the bottom-left red edge further to the left. Then we use different lifts of two of the faces, extending the fundamental domain even further to the top-right and bottom-left, as seen in Figure~\ref{fig:T5} (bottom). 

This realises a permissible boundary with positive diagonals. Here, $\alpha$ is again parallel to the red edge and hence has slope $\myfrac{1}{0}$, and we can see that the $\beta$ slope is $\myfrac{-5}{1}$, with respect to the $\cusp_1$ framing. The gluing information for this triangulation is presented below, with X, Y, A, B, C, D again used to indicate which faces are which in the two boundaries. 

\vspace{2mm}
\begin{adjustbox}{scale=0.9,center}
\renewcommand{\arraystretch}{1.2}
\begin{tabular}{c|c|c|c|c}
    $\Hat{\mathbf{T}}_5$ & Face 012 & Face 013 & Face 023 & Face 123 \\
    \hline
    0 & D & C & 1(320) & X \\
    1 & A & B & 0(320) & Y 
\end{tabular}
\end{adjustbox}

\section{Counting tetrahedra}\label{sec:counting-tets}
In this section we verify that the five triangulations described in Section~\ref{sec:triangulations} account for minimal triangulations of all the census knots listed in Appendix~\ref{app:census-list} (apart from the figure-8 knot). 

Recall that a triangulation denoted by $\mathcal{T}\big(\mathcal{S}_1,\mathcal{S}_2;\, t_0\big)$ has $t_0$ core tetrahedra and two one-vertex boundaries with slope information $\mathcal{S}_i$ corresponding to each cusp $\cusp_i$. When $\mathcal{S}_1$ is replaced by slope information $\Hat{\mathcal{S}}_1$, this indicates that $\cusp_1$ has a permissible boundary and can be Dehn filled using a layered chain. 

Let $t_1$ be the number of tetrahedra required to build either a layered solid torus or a layered chain with primary filling slope $\myfrac{r}{s}$ from a standard or permissible boundary with slope information $\mathcal{S}_1$ or $\Hat{\mathcal{S}}_1$, respectively. Define $t_2$ similarly for the layered solid torus that realises the secondary filling slope $\myfrac{t}{u}$ in $\cusp_2$, with boundary slopes $\mathcal{S}_2$. The total size of any triangulation filled in this way is then $t_0+t_1+t_2$. 

\subsection{Distances in the Farey diagram}
Again, we state the following without proof. Hatcher's book~\cite{HatcherNumbers} contains the relevant results, albeit with slightly different conventions.
\begin{fact}\label{fact:CF-distance} Continued fractions encode distances in the Farey diagram.
    \begin{enumerate}[label=\roman*.]
        \item Every rational number $\,\myfrac{p}{q}>0$ has a unique continued fraction of the form $\CF{a_0,a_1, \cdots, a_{m-1}, a_m}$ with $a_0\geq 0$, all $a_1,\cdots,a_{m-1}$ strictly positive and $a_m>1$. \label{fact:unique-pos-CF}
        \item There is a unique point $\{\myfrac{p}{q},\myfrac{a}{b},\myfrac{c}{d}\}$ in the Farey diagram with $a,b,c,d\geq 0$ and $\dfrac{p}{q}=\dfrac{a+c}{b+d}$.
        \item The length of the path between the points $\{\myfrac{p}{q},\myfrac{a}{b},\myfrac{c}{d}\}$ and $\{0,-1,\infty\}$ is equal to the sum of the coefficients $\sum_{i=0}^m a_i$.
    \end{enumerate}
\end{fact} 

We can use the symmetry of the Farey diagram for negative $\myfrac{p}{q}$, so it is sufficient to consider the positive continued fraction for $\abs{\,\myfrac{p}{q}}$.
\begin{definition}
    For any rational number $\myfrac{p}{q}\in\mathbb{Q}$, take the positive continued fraction for its absolute value $\abs{\,\myfrac{p}{q}}=\CF{a_0;\,a_1, \cdots,\,a_m}$ and define the \emph{norm}
    \[\big\Vert\,\myfrac{p}{q}\big\Vert=\sum_{i=0}^m \,\abs{a_i}.\]
\end{definition}

\subsection{Counting tetrahedra in layered solid tori.} 
Recall from Section~\ref{sec:triangulated-dehn-fillings} that the construction of a layered solid torus is captured by a path in the Farey diagram, starting at the point associated to the triple of boundary slopes, and ending at the nearest point associated to the slope of the desired Dehn filling. For each step in this path, one tetrahedron is added, except at the last step, where a fold is performed instead. 

Recall the boundary triples $\mathcal{P}$, $\mathcal{Q}$, $\mathcal{R}$ and $\mathcal{R}'$, corresponding to the points labelled in Figure~\ref{fig:LST-Farey-paths}.
\begin{proposition}\label{prop:LST-count}
    The number of tetrahedra required to triangulate a $\myfrac{p}{q}$ Dehn filling using a layered solid torus in a one-vertex torus boundary depends on the slopes of the boundary edges and the interval $\myfrac{p}{q}$ lies in. 
    For boundary slopes $\mathcal{P}$, $\mathcal{Q}$, $\mathcal{R}$ or $\mathcal{R}'$, the number of tetrahedra is $\norm{\,\myfrac{p}{q}}+a$, with $a\in\mathbb{Z}$ as given below. 
    
\begin{adjustbox}{scale=0.9,center}
\begin{tikzpicture}[x = 12mm, y = 7mm, node distance = 0mm, outer sep = 0mm]
\foreach \x in {-3,...,1}{
    \draw[lightgray] (\x,-0.5) -- (\x,3.6);}
\draw[black] (-3.2,3.6) -- (1.2,3.6);
\node at (-3,4.2) {$-\frac{1}{0}$};
\node at (-2,4.2) {$-\frac{2}{1}$};
\node at (-1,4.2) {$-\frac{1}{1}$};
\node at (0,4.2) {$\frac{0}{1}$};
\node at (1,4.2) {$\frac{1}{0}$};
\node at (-3.3,3) {$\mathcal{R}$};
    \node at (-2.5,3) {$-1$};
    \node at (-1.5,3) {$-1$};
    \node at (-0.5,3) {$-1$};
    \node at (0.5,3) {$-2$};
\node at (-3.3,2) {$\mathcal{R}'$};
    \node at (-2.5,2) {$-2$};
    \node at (-1.5,2) {$-2$};
    \node at (-0.5,2) {$-2$};
    \node at (0.5,2) {$-1$};
\node at (-3.3,1) {$\mathcal{P}$};
    \node at (-2.5,1) {$-3$};
    \node at (-1.5,1) {$-3$};
    \node at (-0.5,1) {$-1$};
    \node at (0.5,1) {$0$};
\node at (-3.3,0) {$\mathcal{Q}$};
    \node at (-2.5,0) {$-4$};
    \node at (-1.5,0) {$-2$};
    \node at (-0.5,0) {$0$};
    \node at (0.5,0) {$1$};
\end{tikzpicture}
    
\end{adjustbox}
\end{proposition}
\begin{proof}
    Let $\mathcal{S}\left[\,\myfrac{p}{q}\right]$ denote the path in the Farey diagram that begins at the point $\mathcal{S}\in\{\mathcal{R},\mathcal{R}',\mathcal{P},\mathcal{Q}\}$ and ends at the nearest point containing $\myfrac{p}{q}\notin\mathcal{S}$. Then the path $\mathcal{R}'\left[\,\myfrac{p}{q}\right]$ for $\myfrac{p}{q}>0$ has length $\norm{\,\myfrac{p}{q}}$ by part (iii) of Fact~\ref{fact:CF-distance}. By the symmetry of the Farey diagram, the path $\mathcal{R}\left[\,\myfrac{p}{q}\right]$ for $\myfrac{p}{q}<0$ also has length $\norm{\,\myfrac{p}{q}}$. 
    
    Since the number of tetrahedra in a layered solid torus is one less than the length of its Farey path, this accounts for the first three entries in the row labelled $\mathcal{R}$ and the fourth entry in the row labelled $\mathcal{R}'$. 
    
    Figure~\ref{fig:LST-Farey-paths} shows part of the Farey diagram with boundary triples $\mathcal{P},\,\mathcal{Q},\,\mathcal{R}$ and $\mathcal{R}'$ labelled. Selected points on the boundary of $\mathbf{H}^2$ are labelled by a continued fraction expansion. 
    
    The remaining entries in the table can be determined using Figure~\ref{fig:LST-Farey-paths}, by comparing each path type with one of $\mathcal{R}\left[\,\myfrac{p}{q}\right]$ or $\mathcal{R}'\left[\,\myfrac{p}{q}\right]$. \qedhere

    \begin{figure}[htp]
        \centering
        \includegraphics[width=0.6\linewidth]{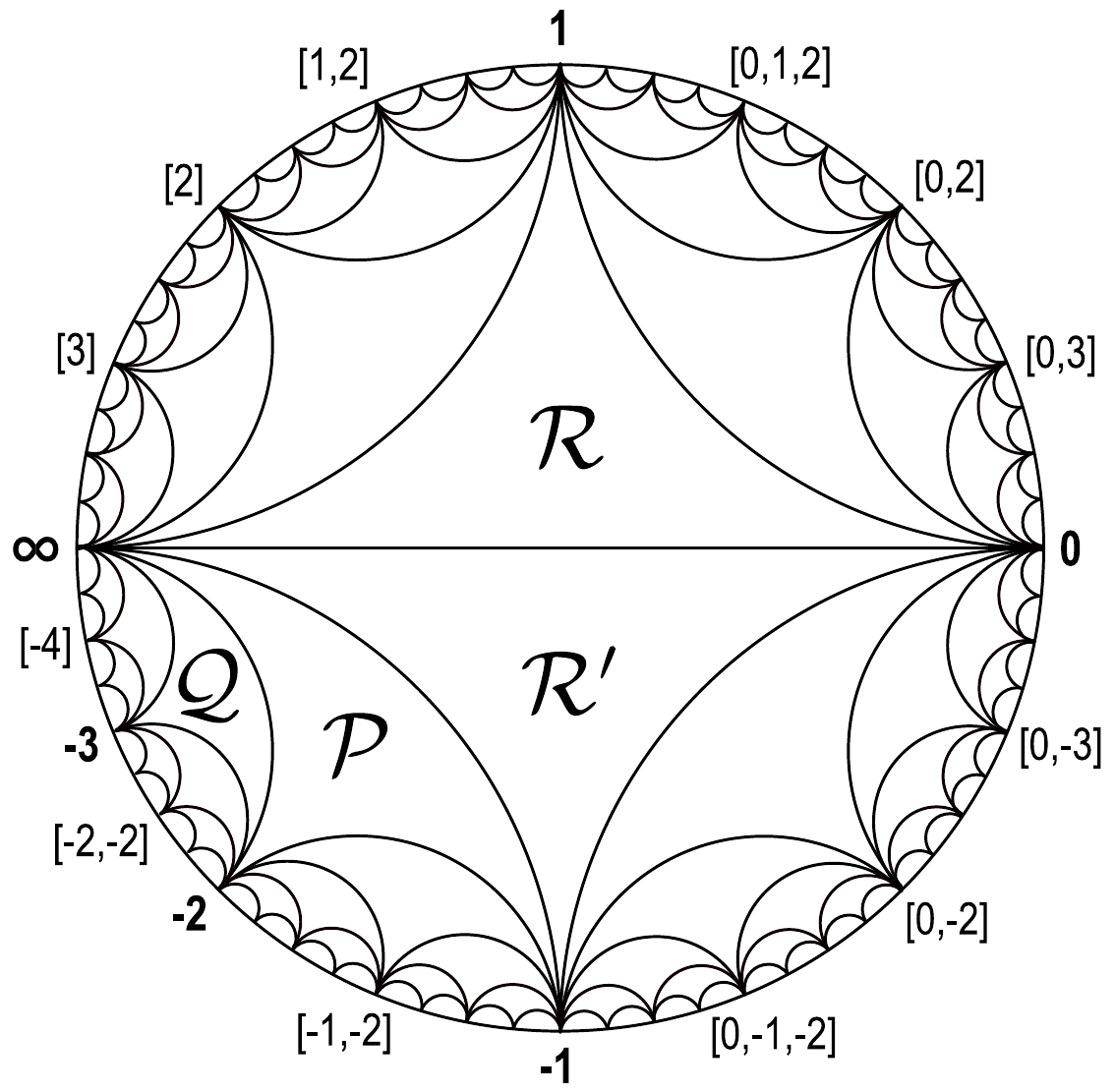}
        \caption{The number of tetrahedra required to Dehn fill a torus boundary with slope information $\mathcal{P},\,\mathcal{Q},\,\mathcal{R}$ or $\mathcal{R}'$ along any rational slope $\myfrac{p}{q}$ can be determined from this diagram.}
        \label{fig:LST-Farey-paths}
    \end{figure}
    \end{proof}

\subsection{Counting tetrahedra in layered chains.}
Next, we consider the number of tetrahedra required to Dehn fill a permissible boundary using a layered chain. Recall the sets of slope information $\Hat{\mathcal{V}}=\{\myfrac{1}{0},\,\myfrac{1}{1};\,-\}$ and $\Hat{\mathcal{U}}=\{\myfrac{1}{0},\myfrac{-5}{1};\,+\}$.
\begin{proposition}\label{prop:chain-size}
    Suppose $k$ is an integer slope in a permissible boundary with slope information $\Hat{\mathcal{S}}\in\{\Hat{\mathcal{U}},\,\Hat{\mathcal{V}}\}$. The number of tetrahedra required to perform a Dehn filling of slope $k$ using a layered chain is
\[ t=\begin{cases}
    |\,k+4\,|,      & \text{if } \Hat{\mathcal{S}}=\Hat{\mathcal{U}}, \text{ or} \\
    |\,k\,|       & \text{if } \Hat{\mathcal{S}}=\Hat{\mathcal{V}}.
\end{cases} \]    
\end{proposition}
\begin{proof}
    Recall Figures~\ref{fig:perm-slopes-a} and~\ref{fig:perm-slopes-b} from Section~\ref{sec:triangulated-dehn-fillings}. If we set $\alpha$ to have slope $\myfrac{1}{0}$ and $\beta$ to have slope $\myfrac{-5}{1}$, then these figures show that we can Dehn fill a $\Hat{\mathcal{U}}$ boundary along slopes $\myfrac{-3}{1},\,\myfrac{-4}{1},\,\myfrac{-5}{1}$ and $\myfrac{-6}{1}$ using layered chains of size 1, 0, 1 and 2, respectively. A single additional tetrahedron is required for each successive integer slope in both directions from this point on. Hence, a filling of slope $k\in\mathbb{Z}$ requires $\abs{k+4}$ tetrahedra.

    Now take a reflected version of Figures~\ref{fig:perm-slopes-a} and~\ref{fig:perm-slopes-b}. Set $\alpha$ to have slope $\myfrac{1}{0}$ again, but this time take $\beta$ to have slope $\myfrac{1}{1}$. From this we deduce that a filling of slope $\myfrac{1}{1}$ will require one tetrahedron, while a filling of slope $\myfrac{0}{1}$ will require no tetrahedra. Extrapolating from here, we see that a filling of slope $k\in\mathbb{Z}$ in a $\Hat{\mathcal{V}}$ boundary requires a layered chain consisting of $\abs{k}$ tetrahedra. 
\end{proof}

\subsection{Identifying a minimal triangulation} 
\begin{observation}
    The figure-8 knot $M_3(\,\myfrac{1}{1},\myfrac{2}{1})=K2_1$ is the only census knot in Appendix~\ref{app:census-list} that is obtained from $M_3$ by two integer fillings. A 3-tetrahedron non-minimal triangulation can be found by filling $\Hat{\mathbf{T}}_4$ with a 1-tetrahedron layered chain in $\cusp_1$ to realise the $\myfrac{1}{1}$ filling, and by folding the $\cusp_2$ boundary closed (i.e. filling with a 0-tetrahedron layered solid torus) to perform the $\myfrac{2}{1}$ filling. The unique minimal triangulation of $K2_1$ can be obtained from this triangulation by performing a 3--2 move.
\end{observation}

\begin{theorem}
The five triangulations described in Section~\ref{sec:triangulations} are sufficient for generating a minimal triangulation of all 228 census knots $K=M_3(\,\myfrac{r}{s},\myfrac{t}{u})\neq K2_1$ listed in Appendix~\ref{app:census-list}, using the methods discussed in Section~\ref{sec:triangulated-dehn-fillings}. 
\end{theorem}
\begin{proof}
    Consider the following partition of the knots in Appendix~\ref{app:census-list} according to their primary slope:
    \begin{enumerate}
        \item Non-integer $\myfrac{r}{s}$ between $-2$ and $0$.
        \item Non-integer $\myfrac{r}{s}$ less than $-2$.
        \item Non-integer $\myfrac{r}{s}$ greater than $0$.
        \item Integer $\myfrac{r}{s}$ greater than $0$.
        \item Integer $\myfrac{r}{s}$ less than $0$.
    \end{enumerate}
    
    The knots described by 1, 2 and 3 are triangulated using $\mathbf{T}_1,\,\mathbf{T}_2$ and $\mathbf{T}_3$, respectively, with Dehn fillings realised by two layered solid tori. 
    The knots described by 4 and 5 are triangulated using $\Hat{\mathbf{T}}_4$ and $\Hat{\mathbf{T}}_5$, respectively, with Dehn fillings realised by a layered solid torus in the one-vertex boundary and a layered chain in the permissible boundary.

    Consider the example entry from Appendix~\ref{app:census-list} shown below. 
\[
\begin{array}{ccc|cc||cc|cc|rc|ccc|c} \renewcommand{\arraycolsep}{1pt}
    \text{Knot} & \text{SnapPea} & \text{Type} & (r,s) & (t,u) & \text{CF}\Bigl\vert\frac{r}{s}\Bigr\vert & \text{CF}\Bigl\vert\frac{t}{u}\Bigr\vert & \Bigl\Vert\frac{r}{s}\Bigr\Vert & \Bigl\Vert\frac{t}{u}\Bigr\Vert & \sim\sfrac{r}{s} & \text{Tri.} & t_0 & t_1 & t_2 & \Sigma \\[1.5mm] \hline\hline
    K4_4 & \text{m}118 & \text{A} & (-1,3) & (-5,2) & \flatCF{0,3} & \flatCF{2,2} & 3 & 4 & -0.3 & \mathbf{T}_2 & 2 & 2 & 0 & 4
\end{array}
\]

For each knot, we provide its pair of filling slopes $\myfrac{r}{s},\,\myfrac{t}{u}$, along with positive continued fractions for (the absolute value of) each, and their respective norms. An approximate value of $\myfrac{r}{s}$ is used to assign a triangulation according to the partitioning above, which in turn determines the value of $t_0$. 

Values for $t_1$ and $t_2$ are determined as in Propositions~\ref{prop:LST-count} and~\ref{prop:chain-size}, then the final column shows the sum $t_0+t_1+t_2$. For the triangulation to be minimal, this value must be equal to the complexity of the knot $C$, as indicated by its name $KC_v$. 

Apart from $K2_1$, every entry in Appendix~\ref{app:census-list} confirms the minimality of the chosen triangulation.
\end{proof}

\subsection{Families of minimal triangulations} 
Having established that we can construct at least one minimal triangulation for each knot, let us consider minimal triangulations of whole families. In the following, we refer to the 42 families of knots described in Section~\ref{sec:families} as the \emph{distinguished families of census knots}.

Recall that a knot determined by one slope in $(-\infty,-2)$ and one slope in $(-2,-1)$ has a minimal triangulation using $\mathbf{T}_2$. However, it is important that the slope in $(-\infty,-2)$ is assigned to the cusp whose associated boundary slopes are $\mathcal{Q}=\{\infty,-3,-2\}$. In particular, for the distinguished families whose primary slope is in $(-2,-1)$, we need to swap which filling is performed in which boundary of $\mathbf{T}_2$. This is possible since the symmetry of $M_3$ allows us to interchange $\cusp_1$ and $\cusp_2$ freely. 

The following array summarises which families with non-integer primary slopes have minimal triangulations using each of the triangulations of $M_3$. Note the three instances of $\mathbf{T}'_2$ indicating where the exchange of cusps $\cusp_1$ and $\cusp_2$ is necessary. 

\vspace{2mm}
\begin{adjustbox}{scale=0.9,center}
    \begin{tikzpicture}[x = 15mm, y = 10mm, node distance = 0mm, outer sep = 0mm]
\foreach \x in {-4,...,1}{
    \draw[lightgray] (\x,-4.3) -- (\x,1.3);}
\foreach \y in {-4,...,1}{
    \draw[lightgray] (-4.3,\y) -- (1.3,\y);}
\node at (-4,-4.7) {$\myfrac{-1}{0}$};
\node at (-3,-4.7) {$\myfrac{-3}{1}$};
\node at (-2,-4.7) {$\myfrac{-2}{1}$};
\node at (-1,-4.7) {$\myfrac{-1}{1}$};
\node at (0,-4.7) {$\myfrac{0}{1}$};
\node at (1,-4.7) {$\myfrac{1}{0}$};
\node at (-4.7,-4) {$\myfrac{-1}{0}$};
\node at (-4.7,-3) {$\myfrac{-3}{1}$};
\node at (-4.7,-2) {$\myfrac{-2}{1}$};
\node at (-4.7,-1) {$\myfrac{-1}{1}$};
\node at (-4.6,0) {$\myfrac{0}{1}$};
\node at (-4.6,1) {$\myfrac{1}{0}$};
\node at (0.5,0.5) {$\mathbf{T}_3$};
\node at (-3.5,-0.5) {$\mathbf{T}_2$};
\node at (-2.5,-0.5) {$\mathbf{T}_2$};
\node at (-1.5,-0.5) {$\mathbf{T}_1$};
\node at (-2.5,-1.5) {$\mathbf{T}_2$};
\node at (-0.5,-1.5) {$\mathbf{T}_1$};
\node at (-3.5,-2.5) {$\mathbf{T}_2$};
\node at (-1.5,-2.5) {$\mathbf{T}'_2$};
\node at (-0.5,-2.5) {$\mathbf{T}'_2$};
\node at (-2.5,-3.5) {$\mathbf{T}_2$};
\node at (-0.5,-3.5) {$\mathbf{T}'_2$};
\node at (-1.75,-5.3) {primary slopes};
\node[rotate=90] at (-5.3,-1.75) {secondary slopes};
\end{tikzpicture}

\end{adjustbox}
   
It is also important to note that when an integer secondary slope arises in a family whose primary slope is not an integer, the triangulation $\mathbf{T}_1,\,\mathbf{T}_2$ or $\mathbf{T}_3$ determined by the primary slope will \textit{not} yield a minimal triangulation for that knot. Fortunately, integer slopes only ever arise as secondary slopes for small values of $\abs{n}$.

\subsection{Beyond the census knots} 
Let $\{K(i)\}_{i\in\mathbb{Z}}$ represent the infinite extension of one of the 42 distinguished families of census knots. Recall that the breadth $x$ of the associated family is the index for which both $K(-x)$ and $K(x-1)$ are recognised as a 9-tetrahedron census knot. Let $c(K)$ denote the complexity of the knot $K$, and observe that 
\[ c\big(K(-x)\big)=c\big(K(x-1)\big)=9.\]

Each subsequent knot requires only an extension of the layered solid torus for the secondary slope. In these families, all secondary slopes beyond those that realise the 9-tetrahedron knots require only a single extra step in the Farey diagram. Hence, we have the following result.
\begin{theorem}\label{thm:bounded-complexity}
For each family $\{K(i)\}_i$ that extends a distinguished family of census knots with breadth $x$, the complexity of each knot, for $j>0$, satisfies 
\[
c\big(K(-x-j)\big)=c\big(K(x-1+j)\big)\leq 9+j.
\]
\end{theorem}

\begin{remark}
    The difference between the triangulations for $K(i)$ and $K(i+j)$ is essentially the same as the \emph{minimal extension} described by Jaco, Rubinstein and Tillmann in Corollary 6 of~\cite{JRT09}. 
\end{remark}

\begin{conjecture}
The inequality in Theorem~\ref{thm:bounded-complexity} can be upgraded to equality for all 42 families of knots in this paper.
\end{conjecture}

\newpage
\appendix
\section{Families of census knots of breadth 4, 3, 2}\label{app:up-to-8}
\begin{adjustbox}{scale=0.9,center}
    \renewcommand{\arraystretch}{1.4}
    $\begin{array}{c|c||cccc|cccc||c}
        (r,s) & (t_n,u_n) & -4 & -3 & -2 & -1 & 0 & 1 & 2 & 3 & M(r/s) \\ \hline 
        (-5,1) & (-1-n,4+5 n) &  K9_{10}   &  K8_{8}   &  K7_{8}   & T_{(2,3)} &  K5_{10}   &  K7_{9}   &  K8_{11}   &  K9_{11}   & \text{m}292 \\ 
        (-4,3) & (-2-3 n,3+4 n) &  K9_{12}   &  K8_{12}   &  K7_{12}   & T_{(2,5)} &  K5_{7}   &  K7_{13}   &  K8_{13}   &  K9_{13}   & \text{m}328 \\ 
        (-5,3) & (-2-3 n,3+5 n) &  K9_{14}   &  K8_{14}   &  K7_{14}   &  K4_{3}   &  K5_{6}   &  K7_{16}   &  K8_{15}   &  K9_{15}   & \text{m}329 \\ 
        (-5,3)_\text{B} & (-12-17 n,5+7 n) &  K9_{16}   &  K8_{16}   &  K7_{15}   &  K3_{1}   &  K5_{16}   &  K7_{17}   &  K8_{17}   &  K9_{17}   & \text{m}329 \\ 
        (3,1) & (1+n,2+3 n) &  K9_{28}   &  K8_{24}   &  K7_{19}   & T_{(2,3)} &  K5_{8}   &  K7_{26}   &  K8_{28}   &  K9_{29}   & \text{s}443 \\ 
        (-7,2)_\text{C} & (-8-13 n,3+5 n) & K9_{32}   &  K8_{30}   &  K7_{29}   &  K4_{4}   &  K5_{17}   &  K7_{31}   &  K8_{32}   &  K9_{33}   & \text{m}366 \\
        (-3,4) & (-3-4 n,2+3 n) &  K9_{52}   &  K8_{38}   &  K7_{21}   & T_{(2,9)} &  K5_{1}   &  K7_{33}   &  K8_{41}   &  K9_{54}   & \text{s}549 \\ 
        (-3,5) & (-3-5 n,2+3 n) &  K9_{73}   &  K8_{52}   &  K7_{32}   & T_{(3,7)} &  K5_{4}   &  K7_{38}   &  K8_{56}   &  K9_{75}   & \text{s}638 \\ \hline
        (-7,3) & (-2-3 n,5+7 n) &&  K9_{23}   &  K8_{21}   &  K4_{3}   &  K6_{16}   &  K8_{25}   &  K9_{25}   && \text{m}357 \\ 
        (-7,3)_\text{B} & (-12-19 n,7+11 n) &&  K9_{24}   &  K8_{23}   &  K5_{15}   &  K6_{15}   &  K8_{26}   &  K9_{27}   && \text{m}357 \\ 
        (-7,2) & (-1-2 n,4+7 n) &&  K9_{30}   &  K8_{27}   &  K5_{14}   &  K6_{11}   &  K8_{29}   &  K9_{31}   && \text{m}366 \\ 
        (-8,3)_\text{C} & (-7-17 n,2+5 n) &&  K9_{41}   &  K8_{34}   &  K5_{17}   &  K6_{18}   &  K8_{36}   &  K9_{43}   && \text{m}388 \\ 
        (-8,3) & (-2-3 n,5+8 n) &&  K9_{44}   &  K8_{35}   &  K5_{18}   &  K6_{21}   &  K8_{37}   &  K9_{45}   && \text{m}388 \\ 
        (1,3) & (2+3 n,1+n) &&  K9_{53}   &  K7_{25}   & T_{(0,1)} &  K5_{9}   &  K8_{39}   &  K9_{57}   && \text{s}578 \\ 
        (3,2) & (1+2 n,2+3 n) &&  K9_{59}   &  K8_{44}   &  K4_{1}   &  K6_{10}   &  K8_{47}   &  K9_{63}   && \text{s}568 \\ 
        (-1,5) & (-4-5 n,1+n) &&  K9_{60}   &  K7_{20}   & T_{(0,1)} &  K5_{1}   &  K8_{40}   &  K9_{71}   && \text{v}1285 \\ 
        (2,3) & (2+3 n,1+2 n) &&  K9_{72}   &  K8_{48}   &  K4_{2}   &  K5_{13}   &  K8_{51}   &  K9_{74}   && \text{s}621 \\ 
        (-3,7) & (-5-7 n,2+3 n) &&  K9_{86}   &  K8_{43}   & T_{(3,11)} &  K6_{6}   &  K8_{58}   &  K9_{89}   && \text{v}1391 \\ 
        (-2,7) & (-4-7 n,1+2 n) &&  K9_{88}   &  K8_{50}   & S_{-3,-2} &  K5_{4}   &  K8_{57}   &  K9_{90}   && \text{v}1411 \\ 
        (-3,8) & (-5-8 n,2+3 n) &&  K9_{91}   &  K8_{55}   & S_{-3,3} &  K6_{7}   &  K8_{63}   &  K9_{92}   && \text{v}1442 \\ \hline
        (-6,1) & (-1-n,5+6 n) &&&  K9_{18}   & T_{(2,3)} &  K7_{20}   &  K9_{26}   &&& \text{s}441 \\ 
        (-5,4) & (-3-4 n,4+5 n) &&&  K9_{34}   & T_{(2,7)} &  K7_{21}   &  K9_{39}   &&& \text{s}506 \\ 
        (4,1) & (1+n,3+4 n) &&&  K9_{37}   & T_{(2,3)} &  K7_{25}   &  K9_{50}   &&& \text{v}1060 \\ 
        (-4,5) & (-4-5 n,3+4 n) &&&  K9_{40}   & T_{(2,11)} &  K7_{12}   &  K9_{56}   &&& \text{v}1204 \\ 
        (-7,4) & (-3-4 n,5+7 n) &&&  K9_{42}   &  K5_{5}   &  K7_{32}   &  K9_{48}   &&& \text{s}503 \\ 
        (-7,4)_\text{B} & (-16-23 n,7+10 n) &&&  K9_{47}   &  K5_{15}   &  K7_{35}   &  K9_{49}   &&& \text{s}503 \\ 
        (-4,7) & (-5-7 n,3+4 n) &&&  K9_{55}   & T_{(3,10)} &  K7_{14}   &  K9_{79}   &&& \text{v}1372 \\ 
        (-7,5) & (-3-5 n,4+7 n) &&&  K9_{62}   &  K6_{13}   &  K7_{33}   &  K9_{64}   &&& \text{s}577 \\ 
        (-8,5) & (-3-5 n,5+8 n) &&&  K9_{65}   &  K6_{12}   &  K7_{38}   &  K9_{70}   &&& \text{s}576 \\ 
        (-10,3)_\text{C} & (-11-19 n,4+7 n) &&&  K9_{67}   &  K6_{18}   &  K7_{36}   &  K9_{68}   &&& \text{s}579 \\ 
        (-5,7) & (-4-7 n,3+5 n) &&&  K9_{76}   &  K6_{3}   &  K7_{13}   &  K9_{84}   &&& \text{v}1351 \\ 
        (-5,8) & (-5-8 n,3+5 n) &&&  K9_{82}   &  K6_{4}   &  K7_{16}   &  K9_{87}   &&& \text{v}1415 \\ 
    \end{array}$
\end{adjustbox}
\section{Suggested diagrams for families of knots}\label{app:knot-diagrams}

\noindent\textbf{Type A knots.} 
We visualise the diagrams of Type A knots using the Hopf model from Figure~\ref{fig:MM-diagrams} (right). A continued fraction for the primary slope, say $\myfrac{r}{s}=\CF{a_0,a_1,\cdots,a_{m-1},a_m}$, gives instructions for a sequence of Dehn twists about annuli parallel to each of $\mu_1$ and $\mu_2$. 

Let $\myfrac{r'}{s'}=\CF{a_0,a_1,\cdots,a_{m-1}}$ and set $a=\abs{r}+\abs{s}$ and $b=\abs{r'}+\abs{s'}$. Then the diagram for this family of knots can be described as a braid on $a$ strands with $b-a\cdot n$ (signed) overstrands and a rational tangle inserted. If $\myfrac{r}{s}>0$, the tangle is a clasp, and if $\myfrac{r}{s}<0$ the tangle is a full twist.  
In the example on the right, $a=7$ and $b=3$, and the clasped tangle is used.

\begin{center}
    \includegraphics[width=0.4\textwidth]{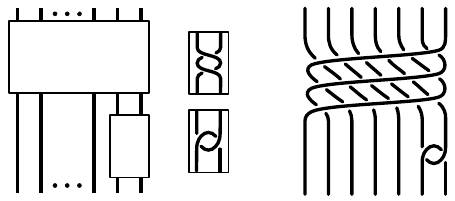}
\end{center}
\vspace{-2mm}
\noindent\textbf{Type B and C knots.} 
These diagrams are found by manipulating the $S^3$ surgery diagrams implied by Theorem~\ref{thm:MP-classification}, with the help of the KLO software~\cite{KirbyCalculator}. The last diagram in each case is two steps away from a final knot diagram, requiring twists about the purple component then twists about the green one. 

\vspace{2mm}
\noindent\textbf{Type B}
\begin{center}
\includegraphics[width=0.8\textwidth]{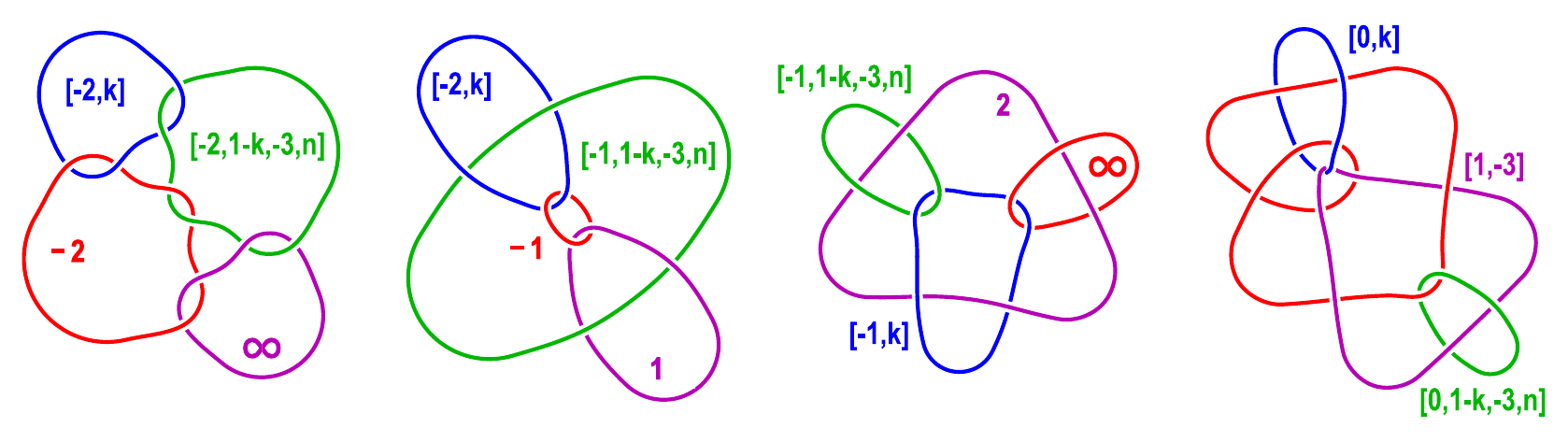}

\includegraphics[width=0.8\textwidth]{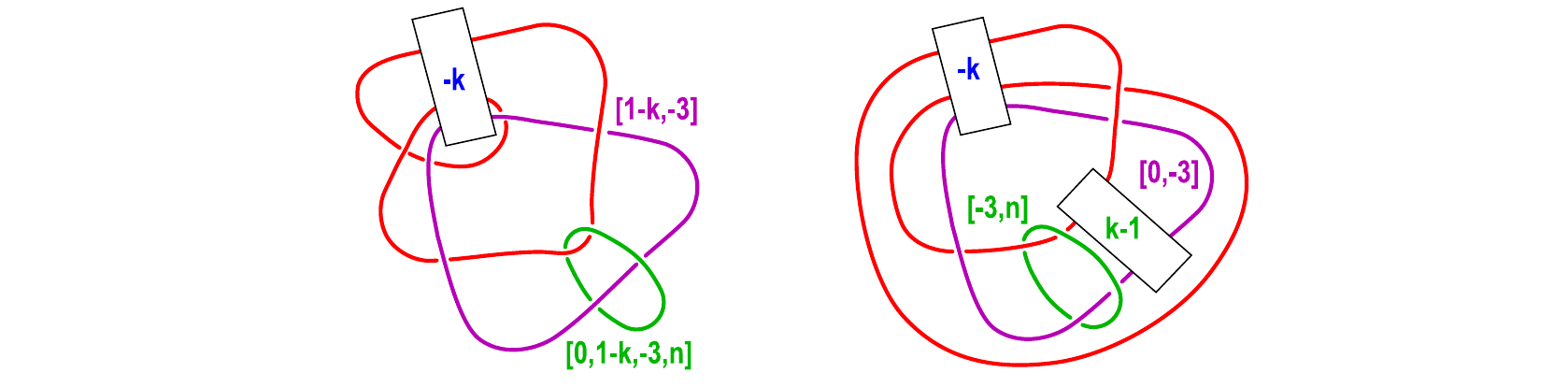}
\end{center}
\vspace{-2mm}
\textbf{Type C}
\begin{center}
\includegraphics[width=0.8\textwidth]{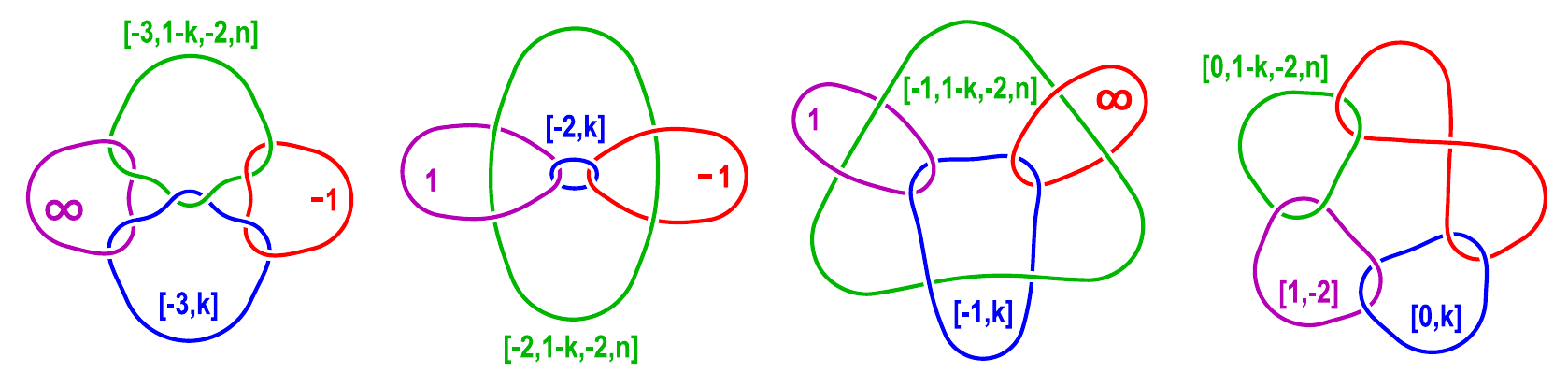}

\includegraphics[width=0.8\textwidth]{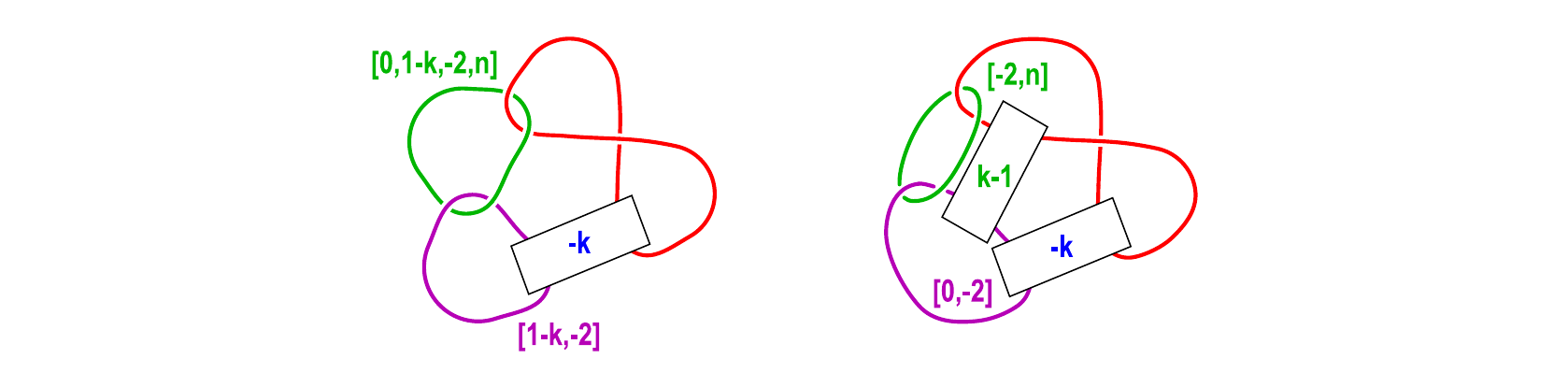}
\end{center}

\section{Census knots as fillings of the magic manifold}\label{app:census-list}

\begin{adjustbox}{center,width=\linewidth}
\renewcommand{\arraycolsep}{3pt}
    $\begin{array}{ccc|cc||cc|cc|rc|ccc|c} 
    \text{Knot} & \text{SnapPea} & \text{Type} & (r,s) & (t,u) & \text{CF}\Bigl\vert\frac{r}{s}\Bigr\vert & \text{CF}\Bigl\vert\frac{t}{u}\Bigr\vert & \Bigl\Vert\frac{r}{s}\Bigr\Vert & \Bigl\Vert\frac{t}{u}\Bigr\Vert & \sim\sfrac{r}{s} & \text{Tri.} & t_0 & t_1 & t_2 & \Sigma \\[1.5mm] \hline\hline  
        K2_{1} & \text{m}004 & \text{A} & (1,1) & (2,1) & \flatCF{1} & \flatCF{2} & 1 & 2 & 1.0 & \Hat{\mathbf{T}}_4 & 1 & 0 & 2 & 3 \\ \midrule
        K3_{1} & \text{m}016 & \text{A} & (-1,2) & (-3,2) & \flatCF{0,2} & \flatCF{1,2} & 2 & 3 & -0.5 & \mathbf{T}_1 & 1 & 0 & 2 & 3 \\ 
        K3_{2} & \text{m}015 & \text{A} & (1,1) & (1,2) & \flatCF{1} & \flatCF{0,2} & 1 & 2 & 1.0 & \Hat{\mathbf{T}}_4 & 1 & 0 & 2 & 3 \\ \midrule
        K4_{1} & \text{m}032 & \text{A} & (1,1) & (3,2) & \flatCF{1} & \flatCF{1,2} & 1 & 3 & 1.0 & \Hat{\mathbf{T}}_4 & 1 & 1 & 2 & 4 \\ 
        K4_{2} & \text{m}053 & \text{A} & (1,1) & (2,3) & \flatCF{1} & \flatCF{0,1,2} & 1 & 3 & 1.0 & \Hat{\mathbf{T}}_4 & 1 & 1 & 2 & 4 \\ 
        K4_{3} & \text{m}082 & \text{A} & (-1,2) & (-5,3) & \flatCF{0,2} & \flatCF{1,1,2} & 2 & 4 & -0.5 & \mathbf{T}_1 & 1 & 1 & 2 & 4 \\ 
        K4_{4} & \text{m}118 & \text{A} & (-5,2) & (-1,3) & \flatCF{2,2} & \flatCF{0,3} & 4 & 3 & -2.5 & \mathbf{T}_2 & 0 & 2 & 2 & 4 \\ \midrule
        K5_{1} & \text{m}071 & \text{A} & (-4,1) & (-1,5) & \flatCF{4} & \flatCF{0,5} & 4 & 5 & -4.0 & \Hat{\mathbf{T}}_5 & 0 & 3 & 2 & 5 \\ 
        K5_{2} & \text{m}074 & \text{A} & (1,1) & (4,3) & \flatCF{1} & \flatCF{1,3} & 1 & 4 & 1.0 & \Hat{\mathbf{T}}_4 & 1 & 2 & 2 & 5 \\ 
        K5_{3} & \text{m}094 & \text{A} & (1,1) & (3,4) & \flatCF{1} & \flatCF{0,1,3} & 1 & 4 & 1.0 & \Hat{\mathbf{T}}_4 & 1 & 2 & 2 & 5 \\ 
        K5_{4} & \text{m}103 & \text{A} & (-4,1) & (-2,7) & \flatCF{4} & \flatCF{0,3,2} & 4 & 5 & -4.0 & \Hat{\mathbf{T}}_5 & 0 & 3 & 2 & 5 \\ 
        K5_{5} & \text{m}144 & \text{A} & (-1,2) & (-7,4) & \flatCF{0,2} & \flatCF{1,1,3} & 2 & 5 & -0.5 & \mathbf{T}_1 & 1 & 2 & 2 & 5 \\ 
        K5_{6} & \text{m}194 & \text{A} & (-2,3) & (-5,3) & \flatCF{0,1,2} & \flatCF{1,1,2} & 3 & 4 & -0.7 & \mathbf{T}_1 & 2 & 1 & 2 & 5 \\ 
        K5_{7} & \text{m}198 & \text{A} & (-2,3) & (-4,3) & \flatCF{0,1,2} & \flatCF{1,3} & 3 & 4 & -0.7 & \mathbf{T}_1 & 2 & 1 & 2 & 5 \\ 
        K5_{8} & \text{m}199 & \text{A} & (3,1) & (1,2) & \flatCF{3} & \flatCF{0,2} & 3 & 2 & 3.0 & \Hat{\mathbf{T}}_4 & 3 & 0 & 2 & 5 \\ 
        K5_{9} & \text{m}201 & \text{A} & (2,1) & (1,3) & \flatCF{2} & \flatCF{0,3} & 2 & 3 & 2.0 & \Hat{\mathbf{T}}_4 & 2 & 1 & 2 & 5 \\ 
        K5_{10} & \text{m}211 & \text{A} & (-5,1) & (-1,4) & \flatCF{5} & \flatCF{0,4} & 5 & 4 & -5.0 & \Hat{\mathbf{T}}_5 & 1 & 2 & 2 & 5 \\ 
        K5_{13} & \text{m}224 & \text{A} & (2,1) & (2,3) & \flatCF{2} & \flatCF{0,2,2} & 2 & 3 & 2.0 & \Hat{\mathbf{T}}_4 & 2 & 1 & 2 & 5 \\ 
        K5_{14} & \text{m}239 & \text{A} & (-7,2) & (-1,3) & \flatCF{3,2} & \flatCF{0,3} & 5 & 3 & -3.5 & \mathbf{T}_2 & 1 & 2 & 2 & 5 \\ 
        K5_{15} & \text{m}240 & \text{B} & (-7,3) & (-7,4) & \flatCF{2,3} & \flatCF{1,1,3} & 5 & 5 & -2.3 & \mathbf{T}_2 & 1 & 2 & 2 & 5 \\ 
        K5_{16} & \text{m}270 & \text{B} & (-12,5) & (-5,3) & \flatCF{2,2,2} & \flatCF{1,1,2} & 6 & 4 & -2.4 & \mathbf{T}_2 & 2 & 1 & 2 & 5 \\ 
        K5_{17} & \text{m}276 & \text{C} & (-7,2) & (-8,3) & \flatCF{3,2} & \flatCF{2,1,2} & 5 & 5 & -3.5 & \mathbf{T}_2 & 2 & 1 & 2 & 5 \\ 
        K5_{18} & \text{m}281 & \text{A} & (-8,3) & (-1,3) & \flatCF{2,1,2} & \flatCF{0,3} & 5 & 3 & -2.7 & \mathbf{T}_2 & 1 & 2 & 2 & 5 \\ \midrule
        K6_{1} & \text{s}016 & \text{A} & (1,1) & (5,4) & \flatCF{1} & \flatCF{1,4} & 1 & 5 & 1.0 & \Hat{\mathbf{T}}_4 & 1 & 3 & 2 & 6 \\ 
        K6_{2} & \text{s}023 & \text{A} & (1,1) & (4,5) & \flatCF{1} & \flatCF{0,1,4} & 1 & 5 & 1.0 & \Hat{\mathbf{T}}_4 & 1 & 3 & 2 & 6 \\ 
        K6_{3} & \text{s}042 & \text{A} & (-4,1) & (-2,9) & \flatCF{4} & \flatCF{0,4,2} & 4 & 6 & -4.0 & \Hat{\mathbf{T}}_5 & 0 & 4 & 2 & 6 \\ 
        K6_{4} & \text{s}068 & \text{A} & (-4,1) & (-3,11) & \flatCF{4} & \flatCF{0,3,1,2} & 4 & 6 & -4.0 & \Hat{\mathbf{T}}_5 & 0 & 4 & 2 & 6 \\ 
        K6_{5} & \text{s}086 & \text{A} & (-1,2) & (-9,5) & \flatCF{0,2} & \flatCF{1,1,4} & 2 & 6 & -0.5 & \mathbf{T}_1 & 1 & 3 & 2 & 6 \\ 
        K6_{6} & \text{s}104 & \text{A} & (-5,2) & (-3,7) & \flatCF{2,2} & \flatCF{0,2,3} & 4 & 5 & -2.5 & \mathbf{T}_2 & 0 & 4 & 2 & 6 \\ 
        K6_{7} & \text{s}114 & \text{A} & (-5,2) & (-3,8) & \flatCF{2,2} & \flatCF{0,2,1,2} & 4 & 5 & -2.5 & \mathbf{T}_2 & 0 & 4 & 2 & 6 \\ 
        K6_{8} & \text{s}188 & \text{A} & (2,1) & (2,5) & \flatCF{2} & \flatCF{0,2,2} & 2 & 4 & 2.0 & \Hat{\mathbf{T}}_4 & 2 & 2 & 2 & 6 \\ 
        K6_{9} & \text{s}194 & \text{A} & (2,1) & (3,5) & \flatCF{2} & \flatCF{0,1,1,2} & 2 & 4 & 2.0 & \Hat{\mathbf{T}}_4 & 2 & 2 & 2 & 6 \\ 
        K6_{10} & \text{s}239 & \text{A} & (1,2) & (3,2) & \flatCF{0,2} & \flatCF{1,2} & 2 & 3 & 0.5 & \mathbf{T}_3 & 0 & 1 & 5 & 6 \\ 
        K6_{11} & \text{s}294 & \text{A} & (-7,2) & (-1,4) & \flatCF{3,2} & \flatCF{0,4} & 5 & 4 & -3.5 & \mathbf{T}_2 & 1 & 3 & 2 & 6 \\ 
        K6_{12} & \text{s}301 & \text{A} & (-2,3) & (-8,5) & \flatCF{0,1,2} & \flatCF{1,1,1,2} & 3 & 5 & -0.7 & \mathbf{T}_1 & 2 & 2 & 2 & 6 \\ 
        K6_{13} & \text{s}308 & \text{A} & (-2,3) & (-7,5) & \flatCF{0,1,2} & \flatCF{1,2,2} & 3 & 5 & -0.7 & \mathbf{T}_1 & 2 & 2 & 2 & 6 \\ 
        K6_{14} & \text{s}336 & \text{A} & (-10,3) & (-1,3) & \flatCF{3,3} & \flatCF{0,3} & 6 & 3 & -3.3 & \mathbf{T}_2 & 2 & 2 & 2 & 6 \\ 
        K6_{15} & \text{s}344 & \text{B} & (-7,3) & (-12,7) & \flatCF{2,3} & \flatCF{1,1,2,2} & 5 & 6 & -2.3 & \mathbf{T}_2 & 1 & 3 & 2 & 6 \\ 
        K6_{16} & \text{s}346 & \text{A} & (-7,3) & (-2,5) & \flatCF{2,3} & \flatCF{0,2,2} & 5 & 4 & -2.3 & \mathbf{T}_2 & 1 & 3 & 2 & 6 \\ 
        K6_{17} & \text{s}367 & \text{A} & (-11,4) & (-1,3) & \flatCF{2,1,3} & \flatCF{0,3} & 6 & 3 & -2.8 & \mathbf{T}_2 & 2 & 2 & 2 & 6 \\ 
        K6_{18} & \text{s}369 & \text{C} & (-10,3) & (-8,3) & \flatCF{3,3} & \flatCF{2,1,2} & 6 & 5 & -3.3 & \mathbf{T}_2 & 3 & 1 & 2 & 6 \\ 
        K6_{21} & \text{s}407 & \text{A} & (-8,3) & (-2,5) & \flatCF{2,1,2} & \flatCF{0,2,2} & 5 & 4 & -2.7 & \mathbf{T}_2 & 1 & 3 & 2 & 6 \\ 
        \end{array}$
\end{adjustbox}

\begin{adjustbox}{center,width=\linewidth}
\renewcommand{\arraycolsep}{3pt}
    $\begin{array}{ccc|cc||cc|cc|rc|ccc|c} 
    \text{Knot} & \text{SnapPea} & \text{Type} & (r,s) & (t,u) & \text{CF}\Bigl\vert\frac{r}{s}\Bigr\vert & \text{CF}\Bigl\vert\frac{t}{u}\Bigr\vert & \Bigl\Vert\frac{r}{s}\Bigr\Vert & \Bigl\Vert\frac{t}{u}\Bigr\Vert & \sim\sfrac{r}{s} & \text{Tri.} & t_0 & t_1 & t_2 & \Sigma \\[1.5mm] \hline\hline   
        K7_{1} & \text{v}0016 & \text{A} & (1,1) & (6,5) & \flatCF{1} & \flatCF{1,5} & 1 & 6 & 1.0 & \Hat{\mathbf{T}}_4 & 1 & 4 & 2 & 7 \\ 
        K7_{2} & \text{v}0025 & \text{A} & (1,1) & (5,6) & \flatCF{1} & \flatCF{0,1,5} & 1 & 6 & 1.0 & \Hat{\mathbf{T}}_4 & 1 & 4 & 2 & 7 \\ 
        K7_{3} & \text{v}0082 & \text{A} & (-4,1) & (-3,13) & \flatCF{4} & \flatCF{0,4,3} & 4 & 7 & -4.0 & \Hat{\mathbf{T}}_5 & 0 & 5 & 2 & 7 \\ 
        K7_{4} & \text{v}0114 & \text{A} & (-4,1) & (-4,15) & \flatCF{4} & \flatCF{0,3,1,3} & 4 & 7 & -4.0 & \Hat{\mathbf{T}}_5 & 0 & 5 & 2 & 7 \\ 
        K7_{5} & \text{v}0165 & \text{A} & (-1,2) & (-11,6) & \flatCF{0,2} & \flatCF{1,1,5} & 2 & 7 & -0.5 & \mathbf{T}_1 & 1 & 4 & 2 & 7 \\ 
        K7_{6} & \text{v}0220 & \text{A} & (-5,2) & (-5,12) & \flatCF{2,2} & \flatCF{0,2,2,2} & 4 & 6 & -2.5 & \mathbf{T}_2 & 0 & 5 & 2 & 7 \\ 
        K7_{7} & \text{v}0223 & \text{A} & (-5,2) & (-5,13) & \flatCF{2,2} & \flatCF{0,2,1,1,2} & 4 & 6 & -2.5 & \mathbf{T}_2 & 0 & 5 & 2 & 7 \\ 
        K7_{8} & \text{v}0249 & \text{A} & (-5,1) & (-1,6) & \flatCF{5} & \flatCF{0,6} & 5 & 6 & -5.0 & \Hat{\mathbf{T}}_5 & 1 & 4 & 2 & 7 \\ 
        K7_{9} & \text{v}0319 & \text{A} & (-5,1) & (-2,9) & \flatCF{5} & \flatCF{0,4,2} & 5 & 6 & -5.0 & \Hat{\mathbf{T}}_5 & 1 & 4 & 2 & 7 \\ 
        K7_{10} & \text{v}0321 & \text{A} & (2,1) & (3,7) & \flatCF{2} & \flatCF{0,2,3} & 2 & 5 & 2.0 & \Hat{\mathbf{T}}_4 & 2 & 3 & 2 & 7 \\ 
        K7_{11} & \text{v}0329 & \text{A} & (2,1) & (4,7) & \flatCF{2} & \flatCF{0,1,1,3} & 2 & 5 & 2.0 & \Hat{\mathbf{T}}_4 & 2 & 3 & 2 & 7 \\ 
        K7_{12} & \text{v}0330 & \text{A} & (-4,3) & (-4,5) & \flatCF{1,3} & \flatCF{0,1,4} & 4 & 5 & -1.3 & \mathbf{T}_1 & 1 & 4 & 2 & 7 \\ 
        K7_{13} & \text{v}0398 & \text{A} & (-4,3) & (-5,7) & \flatCF{1,3} & \flatCF{0,1,2,2} & 4 & 5 & -1.3 & \mathbf{T}_1 & 1 & 4 & 2 & 7 \\ 
        K7_{14} & \text{v}0407 & \text{A} & (-5,3) & (-4,7) & \flatCF{1,1,2} & \flatCF{0,1,1,3} & 4 & 5 & -1.7 & \mathbf{T}_1 & 1 & 4 & 2 & 7 \\ 
        K7_{15} & \text{v}0424 & \text{B} & (-22,9) & (-5,3) & \flatCF{2,2,4} & \flatCF{1,1,2} & 8 & 4 & -2.4 & \mathbf{T}_2 & 4 & 1 & 2 & 7 \\ 
        K7_{16} & \text{v}0434 & \text{A} & (-5,3) & (-5,8) & \flatCF{1,1,2} & \flatCF{0,1,1,1,2} & 4 & 5 & -1.7 & \mathbf{T}_1 & 1 & 4 & 2 & 7 \\ 
        K7_{17} & \text{v}0497 & \text{B} & (-29,12) & (-5,3) & \flatCF{2,2,2,2} & \flatCF{1,1,2} & 8 & 4 & -2.4 & \mathbf{T}_2 & 4 & 1 & 2 & 7 \\ 
        K7_{18} & \text{v}0521 & \text{A} & (1,2) & (5,2) & \flatCF{0,2} & \flatCF{2,2} & 2 & 4 & 0.5 & \mathbf{T}_3 & 0 & 2 & 5 & 7 \\ 
        K7_{19} & \text{v}0535 & \text{A} & (3,1) & (1,4) & \flatCF{3} & \flatCF{0,4} & 3 & 4 & 3.0 & \Hat{\mathbf{T}}_4 & 3 & 2 & 2 & 7 \\ 
        K7_{20} & \text{v}0545 & \text{A} & (-6,1) & (-1,5) & \flatCF{6} & \flatCF{0,5} & 6 & 5 & -6.0 & \Hat{\mathbf{T}}_5 & 2 & 3 & 2 & 7 \\ 
        K7_{21} & \text{v}0554 & \text{A} & (-3,4) & (-5,4) & \flatCF{0,1,3} & \flatCF{1,4} & 4 & 5 & -0.8 & \mathbf{T}_1 & 3 & 2 & 2 & 7 \\ 
        K7_{22} & \text{v}0570 & \text{A} & (-2,3) & (-11,7) & \flatCF{0,1,2} & \flatCF{1,1,1,3} & 3 & 6 & -0.7 & \mathbf{T}_1 & 2 & 3 & 2 & 7 \\ 
        K7_{23} & \text{v}0573 & \text{A} & (-2,3) & (-10,7) & \flatCF{0,1,2} & \flatCF{1,2,3} & 3 & 6 & -0.7 & \mathbf{T}_1 & 2 & 3 & 2 & 7 \\ 
        K7_{24} & \text{v}0595 & \text{A} & (1,2) & (5,3) & \flatCF{0,2} & \flatCF{1,1,2} & 2 & 4 & 0.5 & \mathbf{T}_3 & 0 & 2 & 5 & 7 \\ 
        K7_{25} & \text{v}0600 & \text{A} & (4,1) & (1,3) & \flatCF{4} & \flatCF{0,3} & 4 & 3 & 4.0 & \Hat{\mathbf{T}}_4 & 4 & 1 & 2 & 7 \\ 
        K7_{26} & \text{v}0656 & \text{A} & (3,1) & (2,5) & \flatCF{3} & \flatCF{0,3,2} & 3 & 4 & 3.0 & \Hat{\mathbf{T}}_4 & 3 & 2 & 2 & 7 \\ 
        K7_{27} & \text{v}0707 & \text{A} & (-13,4) & (-1,3) & \flatCF{3,4} & \flatCF{0,3} & 7 & 3 & -3.3 & \mathbf{T}_2 & 3 & 2 & 2 & 7 \\ 
        K7_{28} & \text{v}0709 & \text{A} & (-9,4) & (-9,5) & \flatCF{2,4} & \flatCF{1,1,4} & 6 & 6 & -2.3 & \mathbf{T}_2 & 2 & 3 & 2 & 7 \\ 
        K7_{29} & \text{v}0715 & \text{C} & (-7,2) & (-18,7) & \flatCF{3,2} & \flatCF{2,1,1,3} & 5 & 7 & -3.5 & \mathbf{T}_2 & 2 & 3 & 2 & 7 \\ 
        K7_{30} & \text{v}0740 & \text{A} & (-14,5) & (-1,3) & \flatCF{2,1,4} & \flatCF{0,3} & 7 & 3 & -2.8 & \mathbf{T}_2 & 3 & 2 & 2 & 7 \\ 
        K7_{31} & \text{v}0741 & \text{C} & (-7,2) & (-21,8) & \flatCF{3,2} & \flatCF{2,1,1,1,2} & 5 & 7 & -3.5 & \mathbf{T}_2 & 2 & 3 & 2 & 7 \\ 
        K7_{32} & \text{v}0759 & \text{A} & (-3,5) & (-7,4) & \flatCF{0,1,1,2} & \flatCF{1,1,3} & 4 & 5 & -0.6 & \mathbf{T}_1 & 3 & 2 & 2 & 7 \\ 
        K7_{33} & \text{v}0765 & \text{A} & (-3,4) & (-7,5) & \flatCF{0,1,3} & \flatCF{1,2,2} & 4 & 5 & -0.8 & \mathbf{T}_1 & 3 & 2 & 2 & 7 \\ 
        K7_{34} & \text{v}0830 & \text{A} & (-9,2) & (-1,4) & \flatCF{4,2} & \flatCF{0,4} & 6 & 4 & -4.5 & \mathbf{T}_2 & 2 & 3 & 2 & 7 \\ 
        K7_{35} & \text{v}0847 & \text{B} & (-16,7) & (-7,4) & \flatCF{2,3,2} & \flatCF{1,1,3} & 7 & 5 & -2.3 & \mathbf{T}_2 & 3 & 2 & 2 & 7 \\ 
        K7_{36} & \text{v}0912 & \text{C} & (-10,3) & (-11,4) & \flatCF{3,3} & \flatCF{2,1,3} & 6 & 6 & -3.3 & \mathbf{T}_2 & 3 & 2 & 2 & 7 \\ 
        K7_{37} & \text{v}0939 & \text{A} & (-11,3) & (-1,4) & \flatCF{3,1,2} & \flatCF{0,4} & 6 & 4 & -3.7 & \mathbf{T}_2 & 2 & 3 & 2 & 7 \\ 
        K7_{38} & \text{v}0945 & \text{A} & (-3,5) & (-8,5) & \flatCF{0,1,1,2} & \flatCF{1,1,1,2} & 4 & 5 & -0.6 & \mathbf{T}_1 & 3 & 2 & 2 & 7 \\ 
        K7_{42} & \text{v}1077 & \text{A} & (-12,5) & (-2,5) & \flatCF{2,2,2} & \flatCF{0,2,2} & 6 & 4 & -2.4 & \mathbf{T}_2 & 2 & 3 & 2 & 7 \\ 
        K7_{43} & \text{v}1109 & \text{A} & (-13,5) & (-2,5) & \flatCF{2,1,1,2} & \flatCF{0,2,2} & 6 & 4 & -2.6 & \mathbf{T}_2 & 2 & 3 & 2 & 7 \\ \midrule
        K8_{1} & \text{t}00017 & \text{A} & (1,1) & (7,6) & \flatCF{1} & \flatCF{1,6} & 1 & 7 & 1.0 & \Hat{\mathbf{T}}_4 & 1 & 5 & 2 & 8 \\ 
        K8_{2} & \text{t}00027 & \text{A} & (1,1) & (6,7) & \flatCF{1} & \flatCF{0,1,6} & 1 & 7 & 1.0 & \Hat{\mathbf{T}}_4 & 1 & 5 & 2 & 8 \\ 
        K8_{3} & \text{t}00110 & \text{A} & (-4,1) & (-4,17) & \flatCF{4} & \flatCF{0,4,4} & 4 & 8 & -4.0 & \Hat{\mathbf{T}}_5 & 0 & 6 & 2 & 8 \\ 
        K8_{4} & \text{t}00146 & \text{A} & (-4,1) & (-5,19) & \flatCF{4} & \flatCF{0,3,1,4} & 4 & 8 & -4.0 & \Hat{\mathbf{T}}_5 & 0 & 6 & 2 & 8 \\ 
        K8_{5} & \text{t}00324 & \text{A} & (-1,2) & (-13,7) & \flatCF{0,2} & \flatCF{1,1,6} & 2 & 8 & -0.5 & \mathbf{T}_1 & 1 & 5 & 2 & 8 \\ 
        K8_{6} & \text{t}00423 & \text{A} & (-5,2) & (-7,17) & \flatCF{2,2} & \flatCF{0,2,2,3} & 4 & 7 & -2.5 & \mathbf{T}_2 & 0 & 6 & 2 & 8 \\ 
        K8_{7} & \text{t}00434 & \text{A} & (-5,2) & (-7,18) & \flatCF{2,2} & \flatCF{0,2,1,1,3} & 4 & 7 & -2.5 & \mathbf{T}_2 & 0 & 6 & 2 & 8 \\ 
        \end{array}$
\end{adjustbox}

\begin{adjustbox}{center,width=\linewidth}
\renewcommand{\arraycolsep}{3pt}
    $\begin{array}{ccc|cc||cc|cc|rc|ccc|c} 
    \text{Knot} & \text{SnapPea} & \text{Type} & (r,s) & (t,u) & \text{CF}\Bigl\vert\frac{r}{s}\Bigr\vert & \text{CF}\Bigl\vert\frac{t}{u}\Bigr\vert & \Bigl\Vert\frac{r}{s}\Bigr\Vert & \Bigl\Vert\frac{t}{u}\Bigr\Vert & \sim\sfrac{r}{s} & \text{Tri.} & t_0 & t_1 & t_2 & \Sigma \\[1.5mm] \hline\hline
        K8_{8} & \text{t}00550 & \text{A} & (-5,1) & (-2,11) & \flatCF{5} & \flatCF{0,5,2} & 5 & 7 & -5.0 & \Hat{\mathbf{T}}_5 & 1 & 5 & 2 & 8 \\ 
        K8_{9} & \text{t}00565 & \text{A} & (2,1) & (4,9) & \flatCF{2} & \flatCF{0,2,4} & 2 & 6 & 2.0 & \Hat{\mathbf{T}}_4 & 2 & 4 & 2 & 8 \\ 
        K8_{10} & \text{t}00577 & \text{A} & (2,1) & (5,9) & \flatCF{2} & \flatCF{0,1,1,4} & 2 & 6 & 2.0 & \Hat{\mathbf{T}}_4 & 2 & 4 & 2 & 8 \\ 
        K8_{11} & \text{t}00621 & \text{A} & (-5,1) & (-3,14) & \flatCF{5} & \flatCF{0,4,1,2} & 5 & 7 & -5.0 & \Hat{\mathbf{T}}_5 & 1 & 5 & 2 & 8 \\ 
        K8_{12} & \text{t}00729 & \text{A} & (-4,3) & (-7,9) & \flatCF{1,3} & \flatCF{0,1,3,2} & 4 & 6 & -1.3 & \mathbf{T}_1 & 1 & 5 & 2 & 8 \\ 
        K8_{13} & \text{t}00787 & \text{A} & (-4,3) & (-8,11) & \flatCF{1,3} & \flatCF{0,1,2,1,2} & 4 & 6 & -1.3 & \mathbf{T}_1 & 1 & 5 & 2 & 8 \\ 
        K8_{14} & \text{t}00826 & \text{A} & (-5,3) & (-7,12) & \flatCF{1,1,2} & \flatCF{0,1,1,2,2} & 4 & 6 & -1.7 & \mathbf{T}_1 & 1 & 5 & 2 & 8 \\ 
        K8_{15} & \text{t}00855 & \text{A} & (-5,3) & (-8,13) & \flatCF{1,1,2} & \flatCF{0,1,1,1,1,2} & 4 & 6 & -1.7 & \mathbf{T}_1 & 1 & 5 & 2 & 8 \\ 
        K8_{16} & \text{t}00873 & \text{B} & (-39,16) & (-5,3) & \flatCF{2,2,3,2} & \flatCF{1,1,2} & 9 & 4 & -2.4 & \mathbf{T}_2 & 5 & 1 & 2 & 8 \\ 
        K8_{17} & \text{t}00932 & \text{B} & (-46,19) & (-5,3) & \flatCF{2,2,2,1,2} & \flatCF{1,1,2} & 9 & 4 & -2.4 & \mathbf{T}_2 & 5 & 1 & 2 & 8 \\ 
        K8_{18} & \text{t}01033 & \text{A} & (-2,3) & (-14,9) & \flatCF{0,1,2} & \flatCF{1,1,1,4} & 3 & 7 & -0.7 & \mathbf{T}_1 & 2 & 4 & 2 & 8 \\ 
        K8_{19} & \text{t}01037 & \text{A} & (-2,3) & (-13,9) & \flatCF{0,1,2} & \flatCF{1,2,4} & 3 & 7 & -0.7 & \mathbf{T}_1 & 2 & 4 & 2 & 8 \\ 
        K8_{20} & \text{t}01039 & \text{A} & (1,2) & (7,3) & \flatCF{0,2} & \flatCF{2,3} & 2 & 5 & 0.5 & \mathbf{T}_3 & 0 & 3 & 5 & 8 \\ 
        K8_{21} & \text{t}01125 & \text{A} & (-7,3) & (-4,9) & \flatCF{2,3} & \flatCF{0,2,4} & 5 & 6 & -2.3 & \mathbf{T}_2 & 1 & 5 & 2 & 8 \\ 
        K8_{22} & \text{t}01142 & \text{A} & (1,2) & (7,4) & \flatCF{0,2} & \flatCF{1,1,3} & 2 & 5 & 0.5 & \mathbf{T}_3 & 0 & 3 & 5 & 8 \\ 
        K8_{23} & \text{t}01216 & \text{B} & (-7,3) & (-26,15) & \flatCF{2,3} & \flatCF{1,1,2,1,3} & 5 & 8 & -2.3 & \mathbf{T}_2 & 1 & 5 & 2 & 8 \\ 
        K8_{24} & \text{t}01235 & \text{A} & (3,1) & (2,7) & \flatCF{3} & \flatCF{0,3,2} & 3 & 5 & 3.0 & \Hat{\mathbf{T}}_4 & 3 & 3 & 2 & 8 \\ 
        K8_{25} & \text{t}01268 & \text{A} & (-7,3) & (-5,12) & \flatCF{2,3} & \flatCF{0,2,2,2} & 5 & 6 & -2.3 & \mathbf{T}_2 & 1 & 5 & 2 & 8 \\ 
        K8_{26} & \text{t}01292 & \text{B} & (-7,3) & (-31,18) & \flatCF{2,3} & \flatCF{1,1,2,1,1,2} & 5 & 8 & -2.3 & \mathbf{T}_2 & 1 & 5 & 2 & 8 \\ 
        K8_{27} & \text{t}01318 & \text{A} & (-7,2) & (-3,10) & \flatCF{3,2} & \flatCF{0,3,3} & 5 & 6 & -3.5 & \mathbf{T}_2 & 1 & 5 & 2 & 8 \\ 
        K8_{28} & \text{t}01342 & \text{A} & (3,1) & (3,8) & \flatCF{3} & \flatCF{0,2,1,2} & 3 & 5 & 3.0 & \Hat{\mathbf{T}}_4 & 3 & 3 & 2 & 8 \\ 
        K8_{29} & \text{t}01368 & \text{A} & (-7,2) & (-3,11) & \flatCF{3,2} & \flatCF{0,3,1,2} & 5 & 6 & -3.5 & \mathbf{T}_2 & 1 & 5 & 2 & 8 \\ 
        K8_{30} & \text{t}01409 & \text{C} & (-7,2) & (-31,12) & \flatCF{3,2} & \flatCF{2,1,1,2,2} & 5 & 8 & -3.5 & \mathbf{T}_2 & 2 & 4 & 2 & 8 \\ 
        K8_{31} & \text{t}01422 & \text{A} & (-16,5) & (-1,3) & \flatCF{3,5} & \flatCF{0,3} & 8 & 3 & -3.2 & \mathbf{T}_2 & 4 & 2 & 2 & 8 \\ 
        K8_{32} & \text{t}01424 & \text{C} & (-7,2) & (-34,13) & \flatCF{3,2} & \flatCF{2,1,1,1,1,2} & 5 & 8 & -3.5 & \mathbf{T}_2 & 2 & 4 & 2 & 8 \\ 
        K8_{33} & \text{t}01440 & \text{A} & (-17,6) & (-1,3) & \flatCF{2,1,5} & \flatCF{0,3} & 8 & 3 & -2.8 & \mathbf{T}_2 & 4 & 2 & 2 & 8 \\ 
        K8_{34} & \text{t}01598 & \text{C} & (-8,3) & (-24,7) & \flatCF{2,1,2} & \flatCF{3,2,3} & 5 & 8 & -2.7 & \mathbf{T}_2 & 2 & 4 & 2 & 8 \\ 
        K8_{35} & \text{t}01636 & \text{A} & (-8,3) & (-4,11) & \flatCF{2,1,2} & \flatCF{0,2,1,3} & 5 & 6 & -2.7 & \mathbf{T}_2 & 1 & 5 & 2 & 8 \\ 
        K8_{36} & \text{t}01646 & \text{C} & (-8,3) & (-27,8) & \flatCF{2,1,2} & \flatCF{3,2,1,2} & 5 & 8 & -2.7 & \mathbf{T}_2 & 2 & 4 & 2 & 8 \\ 
        K8_{37} & \text{t}01690 & \text{A} & (-8,3) & (-5,13) & \flatCF{2,1,2} & \flatCF{0,2,1,1,2} & 5 & 6 & -2.7 & \mathbf{T}_2 & 1 & 5 & 2 & 8 \\ 
        K8_{38} & \text{t}01757 & \text{A} & (-3,4) & (-9,7) & \flatCF{0,1,3} & \flatCF{1,3,2} & 4 & 6 & -0.8 & \mathbf{T}_1 & 3 & 3 & 2 & 8 \\ 
        K8_{39} & \text{t}01779 & \text{A} & (1,3) & (5,2) & \flatCF{0,3} & \flatCF{2,2} & 3 & 4 & 0.3 & \mathbf{T}_3 & 1 & 2 & 5 & 8 \\ 
        K8_{40} & \text{t}01815 & \text{A} & (-9,2) & (-1,5) & \flatCF{4,2} & \flatCF{0,5} & 6 & 5 & -4.5 & \mathbf{T}_2 & 2 & 4 & 2 & 8 \\ 
        K8_{41} & \text{t}01834 & \text{A} & (-3,4) & (-11,8) & \flatCF{0,1,3} & \flatCF{1,2,1,2} & 4 & 6 & -0.8 & \mathbf{T}_1 & 3 & 3 & 2 & 8 \\ 
        K8_{42} & \text{t}01850 & \text{A} & (-9,4) & (-16,9) & \flatCF{2,4} & \flatCF{1,1,3,2} & 6 & 7 & -2.3 & \mathbf{T}_2 & 2 & 4 & 2 & 8 \\ 
        K8_{43} & \text{t}01863 & \text{A} & (-9,4) & (-3,7) & \flatCF{2,4} & \flatCF{0,2,3} & 6 & 5 & -2.3 & \mathbf{T}_2 & 2 & 4 & 2 & 8 \\ 
        K8_{44} & \text{t}01901 & \text{A} & (3,2) & (3,4) & \flatCF{1,2} & \flatCF{0,3,3} & 3 & 4 & 1.5 & \mathbf{T}_3 & 1 & 2 & 5 & 8 \\ 
        K8_{45} & \text{t}01949 & \text{C} & (-13,4) & (-11,4) & \flatCF{3,4} & \flatCF{2,1,3} & 7 & 6 & -3.3 & \mathbf{T}_2 & 4 & 2 & 2 & 8 \\ 
        K8_{46} & \text{t}01966 & \text{A} & (-13,3) & (-1,4) & \flatCF{4,3} & \flatCF{0,4} & 7 & 4 & -4.3 & \mathbf{T}_2 & 3 & 3 & 2 & 8 \\ 
        K8_{47} & \text{t}02019 & \text{A} & (3,2) & (3,5) & \flatCF{1,2} & \flatCF{0,1,1,2} & 3 & 4 & 1.5 & \mathbf{T}_3 & 1 & 2 & 5 & 8 \\ 
        K8_{48} & \text{t}02069 & \text{A} & (2,3) & (4,3) & \flatCF{0,1,2} & \flatCF{1,3} & 3 & 4 & 0.7 & \mathbf{T}_3 & 1 & 2 & 5 & 8 \\ 
        K8_{49} & \text{t}02099 & \text{A} & (-15,4) & (-1,4) & \flatCF{3,1,3} & \flatCF{0,4} & 7 & 4 & -3.8 & \mathbf{T}_2 & 3 & 3 & 2 & 8 \\ 
        K8_{50} & \text{t}02104 & \text{A} & (-10,3) & (-2,7) & \flatCF{3,3} & \flatCF{0,3,2} & 6 & 5 & -3.3 & \mathbf{T}_2 & 2 & 4 & 2 & 8 \\ 
        K8_{51} & \text{t}02188 & \text{A} & (2,3) & (5,3) & \flatCF{0,1,2} & \flatCF{1,2,2} & 3 & 4 & 0.7 & \mathbf{T}_3 & 1 & 2 & 5 & 8 \\ 
        K8_{52} & \text{t}02238 & \text{A} & (-3,5) & (-12,7) & \flatCF{0,1,1,2} & \flatCF{1,1,2,2} & 4 & 6 & -0.6 & \mathbf{T}_1 & 3 & 3 & 2 & 8 \\ 
        K8_{55} & \text{t}02378 & \text{A} & (-11,4) & (-3,8) & \flatCF{2,1,3} & \flatCF{0,2,1,2} & 6 & 5 & -2.8 & \mathbf{T}_2 & 2 & 4 & 2 & 8 \\ 
        K8_{56} & \text{t}02398 & \text{A} & (-3,5) & (-13,8) & \flatCF{0,1,1,2} & \flatCF{1,1,1,1,2} & 4 & 6 & -0.6 & \mathbf{T}_1 & 3 & 3 & 2 & 8 \\ 
        \end{array}$
\end{adjustbox}

\begin{adjustbox}{center,width=\linewidth}
\renewcommand{\arraycolsep}{3pt}
    $\begin{array}{ccc|cc||cc|cc|rc|ccc|c} 
    \text{Knot} & \text{SnapPea} & \text{Type} & (r,s) & (t,u) & \text{CF}\Bigl\vert\frac{r}{s}\Bigr\vert & \text{CF}\Bigl\vert\frac{t}{u}\Bigr\vert & \Bigl\Vert\frac{r}{s}\Bigr\Vert & \Bigl\Vert\frac{t}{u}\Bigr\Vert & \sim\sfrac{r}{s} & \text{Tri.} & t_0 & t_1 & t_2 & \Sigma \\[1.5mm] \hline\hline
        K8_{57} & \text{t}02404 & \text{A} & (-11,3) & (-2,7) & \flatCF{3,1,2} & \flatCF{0,3,2} & 6 & 5 & -3.7 & \mathbf{T}_2 & 2 & 4 & 2 & 8 \\ 
        K8_{58} & \text{t}02470 & \text{A} & (-12,5) & (-3,7) & \flatCF{2,2,2} & \flatCF{0,2,3} & 6 & 5 & -2.4 & \mathbf{T}_2 & 2 & 4 & 2 & 8 \\ 
        K8_{59} & \text{t}02537 & \text{A} & (-17,7) & (-2,5) & \flatCF{2,2,3} & \flatCF{0,2,2} & 7 & 4 & -2.4 & \mathbf{T}_2 & 3 & 3 & 2 & 8 \\ 
        K8_{60} & \text{t}02567 & \text{A} & (-18,7) & (-2,5) & \flatCF{2,1,1,3} & \flatCF{0,2,2} & 7 & 4 & -2.6 & \mathbf{T}_2 & 3 & 3 & 2 & 8 \\ 
        K8_{63} & \text{t}02639 & \text{A} & (-13,5) & (-3,8) & \flatCF{2,1,1,2} & \flatCF{0,2,1,2} & 6 & 5 & -2.6 & \mathbf{T}_2 & 2 & 4 & 2 & 8 \\ \midrule
        K9_{1} & \text{o}9\_00017 & \text{A} & (1,1) & (8,7) & \flatCF{1} & \flatCF{1,7} & 1 & 8 & 1.0 & \Hat{\mathbf{T}}_4 & 1 & 6 & 2 & 9 \\ 
        K9_{2} & \text{o}9\_00022 & \text{A} & (1,1) & (7,8) & \flatCF{1} & \flatCF{0,1,7} & 1 & 8 & 1.0 & \Hat{\mathbf{T}}_4 & 1 & 6 & 2 & 9 \\ 
        K9_{3} & \text{o}9\_00133 & \text{A} & (-4,1) & (-5,21) & \flatCF{4} & \flatCF{0,4,5} & 4 & 9 & -4.0 & \Hat{\mathbf{T}}_5 & 0 & 7 & 2 & 9 \\ 
        K9_{4} & \text{o}9\_00168 & \text{A} & (-4,1) & (-6,23) & \flatCF{4} & \flatCF{0,3,1,5} & 4 & 9 & -4.0 & \Hat{\mathbf{T}}_5 & 0 & 7 & 2 & 9 \\ 
        K9_{5} & \text{o}9\_00644 & \text{A} & (-1,2) & (-15,8) & \flatCF{0,2} & \flatCF{1,1,7} & 2 & 9 & -0.5 & \mathbf{T}_1 & 1 & 6 & 2 & 9 \\ 
        K9_{6} & \text{o}9\_00797 & \text{A} & (-5,2) & (-9,22) & \flatCF{2,2} & \flatCF{0,2,2,4} & 4 & 8 & -2.5 & \mathbf{T}_2 & 0 & 7 & 2 & 9 \\ 
        K9_{7} & \text{o}9\_00815 & \text{A} & (-5,2) & (-9,23) & \flatCF{2,2} & \flatCF{0,2,1,1,4} & 4 & 8 & -2.5 & \mathbf{T}_2 & 0 & 7 & 2 & 9 \\ 
        K9_{8} & \text{o}9\_01024 & \text{A} & (2,1) & (5,11) & \flatCF{2} & \flatCF{0,2,5} & 2 & 7 & 2.0 & \Hat{\mathbf{T}}_4 & 2 & 5 & 2 & 9 \\ 
        K9_{9} & \text{o}9\_01035 & \text{A} & (2,1) & (6,11) & \flatCF{2} & \flatCF{0,1,1,5} & 2 & 7 & 2.0 & \Hat{\mathbf{T}}_4 & 2 & 5 & 2 & 9 \\ 
        K9_{10} & \text{o}9\_01079 & \text{A} & (-5,1) & (-3,16) & \flatCF{5} & \flatCF{0,5,3} & 5 & 8 & -5.0 & \Hat{\mathbf{T}}_5 & 1 & 6 & 2 & 9 \\ 
        K9_{11} & \text{o}9\_01175 & \text{A} & (-5,1) & (-4,19) & \flatCF{5} & \flatCF{0,4,1,3} & 5 & 8 & -5.0 & \Hat{\mathbf{T}}_5 & 1 & 6 & 2 & 9 \\ 
        K9_{12} & \text{o}9\_01436 & \text{A} & (-4,3) & (-10,13) & \flatCF{1,3} & \flatCF{0,1,3,3} & 4 & 7 & -1.3 & \mathbf{T}_1 & 1 & 6 & 2 & 9 \\ 
        K9_{13} & \text{o}9\_01496 & \text{A} & (-4,3) & (-11,15) & \flatCF{1,3} & \flatCF{0,1,2,1,3} & 4 & 7 & -1.3 & \mathbf{T}_1 & 1 & 6 & 2 & 9 \\ 
        K9_{14} & \text{o}9\_01584 & \text{A} & (-5,3) & (-10,17) & \flatCF{1,1,2} & \flatCF{0,1,1,2,3} & 4 & 7 & -1.7 & \mathbf{T}_1 & 1 & 6 & 2 & 9 \\ 
        K9_{15} & \text{o}9\_01621 & \text{A} & (-5,3) & (-11,18) & \flatCF{1,1,2} & \flatCF{0,1,1,1,1,3} & 4 & 7 & -1.7 & \mathbf{T}_1 & 1 & 6 & 2 & 9 \\ 
        K9_{16} & \text{o}9\_01680 & \text{B} & (-56,23) & (-5,3) & \flatCF{2,2,3,3} & \flatCF{1,1,2} & 10 & 4 & -2.4 & \mathbf{T}_2 & 6 & 1 & 2 & 9 \\ 
        K9_{17} & \text{o}9\_01765 & \text{B} & (-63,26) & (-5,3) & \flatCF{2,2,2,1,3} & \flatCF{1,1,2} & 10 & 4 & -2.4 & \mathbf{T}_2 & 6 & 1 & 2 & 9 \\ 
        K9_{18} & \text{o}9\_01936 & \text{A} & (-6,1) & (-1,7) & \flatCF{6} & \flatCF{0,7} & 6 & 7 & -6.0 & \Hat{\mathbf{T}}_5 & 2 & 5 & 2 & 9 \\ 
        K9_{19} & \text{o}9\_01953 & \text{A} & (-2,3) & (-17,11) & \flatCF{0,1,2} & \flatCF{1,1,1,5} & 3 & 8 & -0.7 & \mathbf{T}_1 & 2 & 5 & 2 & 9 \\ 
        K9_{20} & \text{o}9\_01955 & \text{A} & (-2,3) & (-16,11) & \flatCF{0,1,2} & \flatCF{1,2,5} & 3 & 8 & -0.7 & \mathbf{T}_1 & 2 & 5 & 2 & 9 \\ 
        K9_{21} & \text{o}9\_02030 & \text{A} & (1,2) & (9,4) & \flatCF{0,2} & \flatCF{2,4} & 2 & 6 & 0.5 & \mathbf{T}_3 & 0 & 4 & 5 & 9 \\ 
        K9_{22} & \text{o}9\_02163 & \text{A} & (1,2) & (9,5) & \flatCF{0,2} & \flatCF{1,1,4} & 2 & 6 & 0.5 & \mathbf{T}_3 & 0 & 4 & 5 & 9 \\ 
        K9_{23} & \text{o}9\_02255 & \text{A} & (-7,3) & (-7,16) & \flatCF{2,3} & \flatCF{0,2,3,2} & 5 & 7 & -2.3 & \mathbf{T}_2 & 1 & 6 & 2 & 9 \\ 
        K9_{24} & \text{o}9\_02340 & \text{B} & (-7,3) & (-45,26) & \flatCF{2,3} & \flatCF{1,1,2,1,2,2} & 5 & 9 & -2.3 & \mathbf{T}_2 & 1 & 6 & 2 & 9 \\ 
        K9_{25} & \text{o}9\_02350 & \text{A} & (-7,3) & (-8,19) & \flatCF{2,3} & \flatCF{0,2,2,1,2} & 5 & 7 & -2.3 & \mathbf{T}_2 & 1 & 6 & 2 & 9 \\ 
        K9_{26} & \text{o}9\_02383 & \text{A} & (-6,1) & (-2,11) & \flatCF{6} & \flatCF{0,5,2} & 6 & 7 & -6.0 & \Hat{\mathbf{T}}_5 & 2 & 5 & 2 & 9 \\ 
        K9_{27} & \text{o}9\_02386 & \text{B} & (-7,3) & (-50,29) & \flatCF{2,3} & \flatCF{1,1,2,1,1,1,2} & 5 & 9 & -2.3 & \mathbf{T}_2 & 1 & 6 & 2 & 9 \\ 
        K9_{28} & \text{o}9\_02471 & \text{A} & (3,1) & (3,10) & \flatCF{3} & \flatCF{0,3,3} & 3 & 6 & 3.0 & \Hat{\mathbf{T}}_4 & 3 & 4 & 2 & 9 \\ 
        K9_{29} & \text{o}9\_02559 & \text{A} & (3,1) & (4,11) & \flatCF{3} & \flatCF{0,2,1,3} & 3 & 6 & 3.0 & \Hat{\mathbf{T}}_4 & 3 & 4 & 2 & 9 \\ 
        K9_{30} & \text{o}9\_02655 & \text{A} & (-7,2) & (-5,17) & \flatCF{3,2} & \flatCF{0,3,2,2} & 5 & 7 & -3.5 & \mathbf{T}_2 & 1 & 6 & 2 & 9 \\ 
        K9_{31} & \text{o}9\_02696 & \text{A} & (-7,2) & (-5,18) & \flatCF{3,2} & \flatCF{0,3,1,1,2} & 5 & 7 & -3.5 & \mathbf{T}_2 & 1 & 6 & 2 & 9 \\ 
        K9_{32} & \text{o}9\_02706 & \text{C} & (-7,2) & (-44,17) & \flatCF{3,2} & \flatCF{2,1,1,2,3} & 5 & 9 & -3.5 & \mathbf{T}_2 & 2 & 5 & 2 & 9 \\ 
        K9_{33} & \text{o}9\_02735 & \text{C} & (-7,2) & (-47,18) & \flatCF{3,2} & \flatCF{2,1,1,1,1,3} & 5 & 9 & -3.5 & \mathbf{T}_2 & 2 & 5 & 2 & 9 \\ 
        K9_{34} & \text{o}9\_02772 & \text{A} & (-5,4) & (-5,6) & \flatCF{1,4} & \flatCF{0,1,5} & 5 & 6 & -1.3 & \mathbf{T}_1 & 2 & 5 & 2 & 9 \\ 
        K9_{35} & \text{o}9\_02786 & \text{A} & (-19,6) & (-1,3) & \flatCF{3,6} & \flatCF{0,3} & 9 & 3 & -3.2 & \mathbf{T}_2 & 5 & 2 & 2 & 9 \\ 
        K9_{36} & \text{o}9\_02794 & \text{A} & (-20,7) & (-1,3) & \flatCF{2,1,6} & \flatCF{0,3} & 9 & 3 & -2.9 & \mathbf{T}_2 & 5 & 2 & 2 & 9 \\ 
        K9_{37} & \text{o}9\_02873 & \text{A} & (4,1) & (1,5) & \flatCF{4} & \flatCF{0,5} & 4 & 5 & 4.0 & \Hat{\mathbf{T}}_4 & 4 & 3 & 2 & 9 \\ 
        K9_{38} & \text{o}9\_02909 & \text{A} & (-7,1) & (-1,6) & \flatCF{7} & \flatCF{0,6} & 7 & 6 & -7.0 & \Hat{\mathbf{T}}_5 & 3 & 4 & 2 & 9 \\ 
        K9_{39} & \text{o}9\_03032 & \text{A} & (-5,4) & (-7,9) & \flatCF{1,4} & \flatCF{0,1,3,2} & 5 & 6 & -1.3 & \mathbf{T}_1 & 2 & 5 & 2 & 9 \\ 
        K9_{40} & \text{o}9\_03108 & \text{A} & (-4,5) & (-6,5) & \flatCF{0,1,4} & \flatCF{1,5} & 5 & 6 & -0.8 & \mathbf{T}_1 & 4 & 3 & 2 & 9 \\ 
        K9_{41} & \text{o}9\_03118 & \text{C} & (-8,3) & (-41,12) & \flatCF{2,1,2} & \flatCF{3,2,2,2} & 5 & 9 & -2.7 & \mathbf{T}_2 & 2 & 5 & 2 & 9 \\ 
        K9_{42} & \text{o}9\_03133 & \text{A} & (-7,4) & (-5,9) & \flatCF{1,1,3} & \flatCF{0,1,1,4} & 5 & 6 & -1.8 & \mathbf{T}_1 & 2 & 5 & 2 & 9 \\ 
        \end{array}$
\end{adjustbox}

\begin{adjustbox}{center,width=\linewidth}
\renewcommand{\arraycolsep}{3pt}
    $\begin{array}{ccc|cc||cc|cc|rc|ccc|c} 
    \text{Knot} & \text{SnapPea} & \text{Type} & (r,s) & (t,u) & \text{CF}\Bigl\vert\frac{r}{s}\Bigr\vert & \text{CF}\Bigl\vert\frac{t}{u}\Bigr\vert & \Bigl\Vert\frac{r}{s}\Bigr\Vert & \Bigl\Vert\frac{t}{u}\Bigr\Vert & \sim\sfrac{r}{s} & \text{Tri.} & t_0 & t_1 & t_2 & \Sigma \\[1.5mm] \hline\hline
        K9_{43} & \text{o}9\_03149 & \text{C} & (-8,3) & (-44,13) & \flatCF{2,1,2} & \flatCF{3,2,1,1,2} & 5 & 9 & -2.7 & \mathbf{T}_2 & 2 & 5 & 2 & 9 \\ 
        K9_{44} & \text{o}9\_03162 & \text{A} & (-8,3) & (-7,19) & \flatCF{2,1,2} & \flatCF{0,2,1,2,2} & 5 & 7 & -2.7 & \mathbf{T}_2 & 1 & 6 & 2 & 9 \\ 
        K9_{45} & \text{o}9\_03188 & \text{A} & (-8,3) & (-8,21) & \flatCF{2,1,2} & \flatCF{0,2,1,1,1,2} & 5 & 7 & -2.7 & \mathbf{T}_2 & 1 & 6 & 2 & 9 \\ 
        K9_{46} & \text{o}9\_03278 & \text{A} & (5,1) & (1,4) & \flatCF{5} & \flatCF{0,4} & 5 & 4 & 5.0 & \Hat{\mathbf{T}}_4 & 5 & 2 & 2 & 9 \\ 
        K9_{47} & \text{o}9\_03288 & \text{B} & (-30,13) & (-7,4) & \flatCF{2,3,4} & \flatCF{1,1,3} & 9 & 5 & -2.3 & \mathbf{T}_2 & 5 & 2 & 2 & 9 \\ 
        K9_{48} & \text{o}9\_03313 & \text{A} & (-7,4) & (-7,12) & \flatCF{1,1,3} & \flatCF{0,1,1,2,2} & 5 & 6 & -1.8 & \mathbf{T}_1 & 2 & 5 & 2 & 9 \\ 
        K9_{49} & \text{o}9\_03412 & \text{B} & (-39,17) & (-7,4) & \flatCF{2,3,2,2} & \flatCF{1,1,3} & 9 & 5 & -2.3 & \mathbf{T}_2 & 5 & 2 & 2 & 9 \\ 
        K9_{50} & \text{o}9\_03420 & \text{A} & (4,1) & (2,7) & \flatCF{4} & \flatCF{0,4,2} & 4 & 5 & 4.0 & \Hat{\mathbf{T}}_4 & 4 & 3 & 2 & 9 \\ 
        K9_{51} & \text{o}9\_03526 & \text{A} & (-11,5) & (-11,6) & \flatCF{2,5} & \flatCF{1,1,5} & 7 & 7 & -2.2 & \mathbf{T}_2 & 3 & 4 & 2 & 9 \\ 
        K9_{52} & \text{o}9\_03586 & \text{A} & (-3,4) & (-13,10) & \flatCF{0,1,3} & \flatCF{1,3,3} & 4 & 7 & -0.8 & \mathbf{T}_1 & 3 & 4 & 2 & 9 \\ 
        K9_{53} & \text{o}9\_03597 & \text{A} & (1,3) & (7,2) & \flatCF{0,3} & \flatCF{3,2} & 3 & 5 & 0.3 & \mathbf{T}_3 & 1 & 3 & 5 & 9 \\ 
        K9_{54} & \text{o}9\_03622 & \text{A} & (-3,4) & (-15,11) & \flatCF{0,1,3} & \flatCF{1,2,1,3} & 4 & 7 & -0.8 & \mathbf{T}_1 & 3 & 4 & 2 & 9 \\ 
        K9_{55} & \text{o}9\_03802 & \text{A} & (-4,7) & (-9,5) & \flatCF{0,1,1,3} & \flatCF{1,1,4} & 5 & 6 & -0.6 & \mathbf{T}_1 & 4 & 3 & 2 & 9 \\ 
        K9_{56} & \text{o}9\_03833 & \text{A} & (-4,5) & (-9,7) & \flatCF{0,1,4} & \flatCF{1,3,2} & 5 & 6 & -0.8 & \mathbf{T}_1 & 4 & 3 & 2 & 9 \\ 
        K9_{57} & \text{o}9\_03862 & \text{A} & (1,3) & (8,3) & \flatCF{0,3} & \flatCF{2,1,2} & 3 & 5 & 0.3 & \mathbf{T}_3 & 1 & 3 & 5 & 9 \\ 
        K9_{58} & \text{o}9\_03932 & \text{B} & (-20,9) & (-9,5) & \flatCF{2,4,2} & \flatCF{1,1,4} & 8 & 6 & -2.2 & \mathbf{T}_2 & 4 & 3 & 2 & 9 \\ 
        K9_{59} & \text{o}9\_04037 & \text{A} & (3,2) & (5,7) & \flatCF{1,2} & \flatCF{0,1,2,2} & 3 & 5 & 1.5 & \mathbf{T}_3 & 1 & 3 & 5 & 9 \\ 
        K9_{60} & \text{o}9\_04054 & \text{A} & (-11,2) & (-1,5) & \flatCF{5,2} & \flatCF{0,5} & 7 & 5 & -5.5 & \mathbf{T}_2 & 3 & 4 & 2 & 9 \\ 
        K9_{61} & \text{o}9\_04060 & \text{A} & (-17,4) & (-1,4) & \flatCF{4,4} & \flatCF{0,4} & 8 & 4 & -4.3 & \mathbf{T}_2 & 4 & 3 & 2 & 9 \\ 
        K9_{62} & \text{o}9\_04106 & \text{A} & (-7,5) & (-7,10) & \flatCF{1,2,2} & \flatCF{0,1,2,3} & 5 & 6 & -1.4 & \mathbf{T}_1 & 2 & 5 & 2 & 9 \\ 
        K9_{63} & \text{o}9\_04139 & \text{A} & (3,2) & (5,8) & \flatCF{1,2} & \flatCF{0,1,1,1,2} & 3 & 5 & 1.5 & \mathbf{T}_3 & 1 & 3 & 5 & 9 \\ 
        K9_{64} & \text{o}9\_04205 & \text{A} & (-7,5) & (-8,11) & \flatCF{1,2,2} & \flatCF{0,1,2,1,2} & 5 & 6 & -1.4 & \mathbf{T}_1 & 2 & 5 & 2 & 9 \\ 
        K9_{65} & \text{o}9\_04245 & \text{A} & (-8,5) & (-7,11) & \flatCF{1,1,1,2} & \flatCF{0,1,1,1,3} & 5 & 6 & -1.6 & \mathbf{T}_1 & 2 & 5 & 2 & 9 \\ 
        K9_{66} & \text{o}9\_04269 & \text{A} & (-19,5) & (-1,4) & \flatCF{3,1,4} & \flatCF{0,4} & 8 & 4 & -3.8 & \mathbf{T}_2 & 4 & 3 & 2 & 9 \\ 
        K9_{67} & \text{o}9\_04313 & \text{C} & (-10,3) & (-27,10) & \flatCF{3,3} & \flatCF{2,1,2,3} & 6 & 8 & -3.3 & \mathbf{T}_2 & 3 & 4 & 2 & 9 \\ 
        K9_{68} & \text{o}9\_04431 & \text{C} & (-10,3) & (-30,11) & \flatCF{3,3} & \flatCF{2,1,2,1,2} & 6 & 8 & -3.3 & \mathbf{T}_2 & 3 & 4 & 2 & 9 \\ 
        K9_{69} & \text{o}9\_04435 & \text{C} & (-13,4) & (-14,5) & \flatCF{3,4} & \flatCF{2,1,4} & 7 & 7 & -3.3 & \mathbf{T}_2 & 4 & 3 & 2 & 9 \\ 
        K9_{70} & \text{o}9\_04438 & \text{A} & (-8,5) & (-8,13) & \flatCF{1,1,1,2} & \flatCF{0,1,1,1,1,2} & 5 & 6 & -1.6 & \mathbf{T}_1 & 2 & 5 & 2 & 9 \\ 
        K9_{71} & \text{o}9\_04938 & \text{A} & (-14,3) & (-1,5) & \flatCF{4,1,2} & \flatCF{0,5} & 7 & 5 & -4.7 & \mathbf{T}_2 & 3 & 4 & 2 & 9 \\ 
        K9_{72} & \text{o}9\_04950 & \text{A} & (2,3) & (7,5) & \flatCF{0,1,2} & \flatCF{1,2,2} & 3 & 5 & 0.7 & \mathbf{T}_3 & 1 & 3 & 5 & 9 \\ 
        K9_{73} & \text{o}9\_05021 & \text{A} & (-3,5) & (-17,10) & \flatCF{0,1,1,2} & \flatCF{1,1,2,3} & 4 & 7 & -0.6 & \mathbf{T}_1 & 3 & 4 & 2 & 9 \\ 
        K9_{74} & \text{o}9\_05028 & \text{A} & (2,3) & (8,5) & \flatCF{0,1,2} & \flatCF{1,1,1,2} & 3 & 5 & 0.7 & \mathbf{T}_3 & 1 & 3 & 5 & 9 \\ 
        K9_{75} & \text{o}9\_05177 & \text{A} & (-3,5) & (-18,11) & \flatCF{0,1,1,2} & \flatCF{1,1,1,1,3} & 4 & 7 & -0.6 & \mathbf{T}_1 & 3 & 4 & 2 & 9 \\ 
        K9_{76} & \text{o}9\_05229 & \text{A} & (-5,7) & (-10,7) & \flatCF{0,1,2,2} & \flatCF{1,2,3} & 5 & 6 & -0.7 & \mathbf{T}_1 & 4 & 3 & 2 & 9 \\ 
        K9_{79} & \text{o}9\_05357 & \text{A} & (-4,7) & (-12,7) & \flatCF{0,1,1,3} & \flatCF{1,1,2,2} & 5 & 6 & -0.6 & \mathbf{T}_1 & 4 & 3 & 2 & 9 \\ 
        K9_{80} & \text{o}9\_05426 & \text{A} & (-22,9) & (-2,5) & \flatCF{2,2,4} & \flatCF{0,2,2} & 8 & 4 & -2.4 & \mathbf{T}_2 & 4 & 3 & 2 & 9 \\ 
        K9_{81} & \text{o}9\_05483 & \text{A} & (-23,9) & (-2,5) & \flatCF{2,1,1,4} & \flatCF{0,2,2} & 8 & 4 & -2.6 & \mathbf{T}_2 & 4 & 3 & 2 & 9 \\ 
        K9_{82} & \text{o}9\_05562 & \text{A} & (-5,8) & (-11,7) & \flatCF{0,1,1,1,2} & \flatCF{1,1,1,3} & 5 & 6 & -0.6 & \mathbf{T}_1 & 4 & 3 & 2 & 9 \\ 
        K9_{84} & \text{o}9\_05618 & \text{A} & (-5,7) & (-11,8) & \flatCF{0,1,2,2} & \flatCF{1,2,1,2} & 5 & 6 & -0.7 & \mathbf{T}_1 & 4 & 3 & 2 & 9 \\ 
        K9_{86} & \text{o}9\_05860 & \text{A} & (-16,7) & (-3,7) & \flatCF{2,3,2} & \flatCF{0,2,3} & 7 & 5 & -2.3 & \mathbf{T}_2 & 3 & 4 & 2 & 9 \\ 
        K9_{87} & \text{o}9\_05970 & \text{A} & (-5,8) & (-13,8) & \flatCF{0,1,1,1,2} & \flatCF{1,1,1,1,2} & 5 & 6 & -0.6 & \mathbf{T}_1 & 4 & 3 & 2 & 9 \\ 
        K9_{88} & \text{o}9\_06060 & \text{A} & (-17,5) & (-2,7) & \flatCF{3,2,2} & \flatCF{0,3,2} & 7 & 5 & -3.4 & \mathbf{T}_2 & 3 & 4 & 2 & 9 \\ 
        K9_{89} & \text{o}9\_06128 & \text{A} & (-19,8) & (-3,7) & \flatCF{2,2,1,2} & \flatCF{0,2,3} & 7 & 5 & -2.4 & \mathbf{T}_2 & 3 & 4 & 2 & 9 \\ 
        K9_{90} & \text{o}9\_06154 & \text{A} & (-18,5) & (-2,7) & \flatCF{3,1,1,2} & \flatCF{0,3,2} & 7 & 5 & -3.6 & \mathbf{T}_2 & 3 & 4 & 2 & 9 \\ 
        K9_{91} & \text{o}9\_06248 & \text{A} & (-19,7) & (-3,8) & \flatCF{2,1,2,2} & \flatCF{0,2,1,2} & 7 & 5 & -2.7 & \mathbf{T}_2 & 3 & 4 & 2 & 9 \\ 
        K9_{92} & \text{o}9\_06301 & \text{A} & (-21,8) & (-3,8) & \flatCF{2,1,1,1,2} & \flatCF{0,2,1,2} & 7 & 5 & -2.6 & \mathbf{T}_2 & 3 & 4 & 2 & 9 \\ 
    \end{array}$
\end{adjustbox}
\newpage
\printbibliography
\end{document}